\documentclass{amsart}
\usepackage[all]{xy}
\usepackage{color}
\usepackage{amsthm}
\usepackage{amssymb}
\usepackage[colorlinks=true]{hyperref}

\numberwithin{equation}{section}

\newtheorem{theorem}[equation]{Theorem}
\newtheorem*{theorem*}{Theorem}
\newtheorem{lemma}[equation]{Lemma}

\newtheorem*{conjecture*}{Mamma Conjecture}
\newtheorem*{conjecture1*}{Mamma Conjecture (revisited)}
\newtheorem{proposition}[equation]{Proposition}
\newtheorem{corollary}[equation]{Corollary}
\newtheorem*{corollary*}{Corollary}

\theoremstyle{remark}
\newtheorem{definition}[equation]{Definition}

\newtheorem{example}[equation]{Example}

\theoremstyle{remark}
\newtheorem{remark}[equation]{Remark}

\newcommand{\cA}{{\mathcal A}}
\newcommand{\cB}{{\mathcal B}}
\newcommand{\cC}{{\mathcal C}}
\newcommand{\cD}{{\mathcal D}}

\newcommand{\cJ}{{\mathcal J}}

\newcommand{\cN}{{\mathcal N}}
\newcommand{\cO}{{\mathcal O}}

\newcommand{\cQ}{{\mathcal Q}}

\newcommand{\cS}{{\mathcal S}}
\newcommand{\cT}{{\mathcal T}}
\newcommand{\cU}{{\mathcal U}}

\newcommand{\cZ}{{\mathcal Z}}

\newcommand{\bbG}{\mathbb{G}}

\newcommand{\bbL}{\mathbb{L}}

\newcommand{\bbQ}{\mathbb{Q}}
\newcommand{\bbZ}{\mathbb{Z}}

\DeclareMathOperator{\id}{id}

\DeclareMathOperator{\NChow}{NChow} 
\DeclareMathOperator{\NNum}{NNum} 
\DeclareMathOperator{\NHom}{NHom} 
\DeclareMathOperator{\Chow}{Chow} 
\DeclareMathOperator{\Num}{Num} 
\DeclareMathOperator{\even}{even} 
\DeclareMathOperator{\odd}{odd} 

\newcommand{\dgcat}{\mathsf{dgcat}}

\newcommand{\perf}{\mathsf{perf}}

\newcommand{\dg}{\mathsf{dg}}

\newcommand{\Hom}{\mathrm{Hom}}
\newcommand{\End}{\mathrm{End}}

\newcommand{\rep}{\mathsf{rep}}

\newcommand{\Hmo}{\mathsf{Hmo}}
\newcommand{\op}{\mathsf{op}}

\newcommand{\too}{\longrightarrow}

\newcommand{\ie}{\textsl{i.e.}\ }

\title[Noncommutative motives, Tannakian structures, and Galois groups]{Noncommutative numerical motives, \\Tannakian structures, and \\ motivic Galois groups}
\author{Matilde Marcolli and Gon{\c c}alo~Tabuada}

\address{Matilde Marcolli, Mathematics Department, Mail Code 253-37, Caltech, 1200 E.~California Blvd. Pasadena, CA 91125, USA}
\email{matilde@caltech.edu}

\address{Gon{\c c}alo Tabuada, Department of Mathematics, MIT, Cambridge, MA 02139, USA}
\email{tabuada@math.mit.edu}

\date{\today}


\thanks{The first named author was supported by the NSF grants DMS-0901221 and DMS-1007207. The second named author was supported by the J.H. and E.V. Wade award.}

\begin{document}
\begin{abstract}
In this article we further the study of noncommutative numerical motives, initiated in \cite{Semi,Konts}. By exploring the change-of-coefficients mechanism, we start by improving some of the main results of \cite{Semi}. Then, making use of the notion of Schur-finiteness, we prove that the category $\NNum(k)_F$ of noncommutative numerical motives is (neutral) super-Tannakian. As in the commutative world, $\NNum(k)_F$ is {\em not} Tannakian. In order to solve this problem we promote periodic cyclic homology to a well-defined symmetric monoidal functor $\overline{HP_\ast}$ on the category of noncommutative Chow motives. This allows us to introduce the correct noncommutative analogues $C_{NC}$ and $D_{NC}$ of Grothendieck's standard conjectures $C$ and $D$. Assuming $C_{NC}$, we prove that $\NNum(k)_F$ can be made into a Tannakian category $\NNum^\dagger(k)_F$ by modifying its symmetry isomorphism constraints. By further assuming $D_{NC}$, we neutralize the Tannakian category $\Num^\dagger(k)_F$ using $\overline{HP_\ast}$. Via the (super-)Tannakian formalism, we then obtain well-defined {\em noncommutative motivic (super-)Galois groups}. Finally, making use of Deligne-Milne's theory of Tate triples, we construct explicit homomorphisms relating these new noncommutative motivic (super-)Galois groups with the classical ones.
\end{abstract}

\maketitle
\tableofcontents
\section{Introduction}\label{Intro}

\subsection*{Noncommutative motives}

Over the past two decades Bondal, Drinfeld, Kaledin, Kapranov, Kontsevich, Orlov, Van den Bergh, and others, have been promoting a broad noncommutative (algebraic) geometry program where ``geometry'' is performed directly on dg categories; consult~\cite{Kapranov,BV,BO,Drinfeld,Chitalk,Kaledin,IAS,ENS,Miami,finMot}. In this vein, Kontsevich introduced the category $\NChow(k)_F$ of {\em noncommutative Chow motives} (over a base commutative ring $k$ and with coefficients in a field $F$); see \S\ref{sub:Chow}. Recently, making use of Hochschild homology, the authors introduced the category $\NNum(k)_F$ of {\em noncommutative numerical motives}; see \S\ref{sub:Num}. Under mild assumptions on $k$ and $F$ (see Theorem~\ref{thm:semi}) the category $\NNum(k)_F$ is abelian semi-simple. The precise relation between $\NChow(k)_\bbQ$, $\NNum(k)_\bbQ$, and the classical categories of Chow and numerical motives (with rational coefficients) can be depicted as follows
\begin{equation}\label{eq:diagram}
\xymatrix@C=2em@R=1.5em{
& \Chow(k)_\bbQ  \ar[d]_-\tau \ar[dl] &&  \\
\Num(k)_\bbQ \ar[d]_-\tau & \Chow(k)_\bbQ/_{\!\!-\otimes \bbQ(1)}  \ar[dl]  \ar[rr]^-R && \ar[dl] \NChow(k)_\bbQ  \\
\Num(k)_\bbQ/_{\!\!-\otimes \bbQ(1)} \ar[rr]_{R_{\cN}} & & \NNum(k)_\bbQ & \,.
}
\end{equation}
Here, $\NChow(k)_\bbQ/_{\!\!-\otimes \bbQ(1)}$ and $\Num(k)_\bbQ/_{\!\!-\otimes \bbQ(1)}$ are the orbit categories  associated to the auto-equivalence $-\otimes \bbQ(1)$ (see Appendix~\ref{app:orbit}) and $R$ and $R_\cN$ are fully-faithful functors; consult \cite{Semi} (or the survey \cite{Survey}) for further details.

\subsection*{Motivating questions}

In the commutative world, the category $\Num(k)_F$ of numerical motives is known to be not only abelian semi-simple but also (neutral) super-Tannakian. Moreover, assuming the standard conjecture $C$ (or even the sign conjecture $C^+$), $\Num(k)_F$ can be made into a Tannakian category $\Num^\dagger(k)_F$ by modifying its symmetry isomorphism constraints; see Jannsen~\cite{Jannsen}. By further assuming the standard conjecture $D$, the classical Weil cohomologies can be used to neutralize the Tannakian category $\Num^\dagger(k)_F$; see \cite[\S6]{Andre}. As explained in Appendix~\ref{appendix}, the (super-)Tannakian formalism provides us then with well-defined motivic (super-)Galois groups $\mathrm{sGal}(\Num(k)_F)$ and $\mathrm{Gal}(\Num^\dagger(k)_F)$ encoding deep arithmetic/geometric properties of smooth projective $k$-schemes. This circle of results and conjectures in the commutative world leads us naturally to the following questions in the noncommutative world:

\smallbreak

{\bf Question I:} \textit{Is the category $\NNum(k)_F$  (neutral) super-Tannakian ?} 

{\bf  Question II:} \textit{Does the standard conjecture $C$ (or the sign conjecture $C^+$) admits a noncommutative analogue $C_{NC}$ ? Does $C_{NC}$ allows us to make $\NNum(k)_F$ into a Tannakian category $\NNum^\dagger(k)_F$ ?}

{\bf  Question III:} \textit{Does the standard conjecture $D$ admits a noncommutative analogue $D_{NC}$ ? Does $D_{NC}$ allows us to neutralize $\NNum^\dagger(k)_F$ ?}

{\bf Question IV:} \textit{Assuming that Questions I, II and III have affirmative answers, how do the noncommutative motivic (super-)Galois groups hence obtained relate with $\mathrm{sGal}(\Num(k)_F)$ and $\mathrm{Gal}(\Num^\dagger(k)_F)$   ?}

\subsection*{Statement of results}
By exploring the change-of-coefficients mechanism we start by improving the main results of \cite{Semi} concerning the semi-simplicity of the category $\NNum(k)_F$ and the relation between the commutative and the noncommutative world; see Proposition~\ref{prop:extension} and Theorem~\ref{thm:semi}. Then, making use of the notion of Schur-finiteness (see \S\ref{sec:Schur}), we answer affirmatively {\bf Question I}.
\begin{theorem}\label{thm:main1}
Assume that $F$ is a field of characteristic zero and that $k$ and $F$ are as in Theorem~\ref{thm:semi}.  Then, the category $\NNum(k)_F$ is super-Tannakian. If $F$ is moreover algebraically closed, then $\NNum(k)_F$ is neutral super-Tannakian.
\end{theorem}
Theorem~\ref{thm:main1} (with $F$ algebraically closed) combined with the super-Tannakian formalism gives rise to a super-Galois group scheme $\mathrm{sGal}(\NNum(k)_F)$, which we will call the {\em noncommutative motivic super-Galois group}. Among other consequences, $\NNum(k)_F$ is $\otimes$-equivalent to the category of finite dimensional $F$-valued super-representations of $\mathrm{sGal}(\NNum(k)_F)$.

The category $\NNum(k)_F$ is {\em not} Tannakian since the rank of each one of its objects is not necessarily a non-negative integer. In order to solve this problem, we start by promoting periodic cyclic homology to a well-defined symmetric monoidal functor $\overline{HP_\ast}$ on the category of noncommutative Chow motives; see Theorem~\ref{thm:HP}. Then, given a smooth and proper dg category $\cA$ in the sense of Kontsevich (see \S\ref{sub:smooth}), we formulate the following conjecture:

\smallbreak

{\bf Noncommutative standard conjecture $C_{NC}(\cA)$:} {\it The K{\"u}nneth projectors
$$ \pi_{\cA}^{\pm}:  \overline{HP_\ast}(\cA) \twoheadrightarrow \overline{HP_\ast}^{\pm}(\cA) \hookrightarrow \overline{HP_\ast}(\cA) $$
are algebraic, \ie they can be written as $\pi^{\pm}_{\cA}=\overline{HP_\ast}(\underline{\pi}^{\pm}_{\cA})$, with $\underline{\pi}^{\pm}_{\cA}$ noncommutative correspondences.
}

\smallbreak

As in the commutative world, the noncommutative standard conjecture $C_{NC}$ is stable under tensor products, \ie $ C_{NC}(\cA) + C_{NC}(\cB) \Rightarrow C_{NC}(\cA\otimes_k \cB)$; see Proposition~\ref{prop:stable-tensor}. Its relation with the sign conjecture $C^+$ (see \S\ref{sec:Kunneth}) is the following: given a quasi-compact and quasi-separated $k$-scheme $Z$, it is well-known that the derived category $\cD_\perf(Z)$ of perfect complexes of $\cO_Z$-modules admits a natural dg enhancement $\cD_\perf^\dg(Z)$; consult Lunts-Orlov~\cite{LO} (or \cite[Example~4.5]{CT1}). When $Z$ is moreover smooth and proper, the dg category $\cD_\perf^\dg(Z)$ is smooth and proper in the sense of Kontsevich.
\begin{theorem}\label{thm:equality}
Let $k$ and $F$ be fields of characteristic zero with $k$ a field extension of $F$. Then, $C^+(Z) \Rightarrow C_{NC}(\cD_\perf^\dg(Z))$. 
\end{theorem}
Intuitively speaking, Theorem~\ref{thm:equality} shows us that when restricted to the commutative world, the noncommutative standard conjecture $C_{NC}$ is more likely to hold than the sign conjecture $C^+$ (and therefore than the standard conjecture $C$). Hence, it answers affirmatively to the first part of {\bf Question II}. Our answer to the second part is the following:
\begin{theorem}\label{thm:main2}
Assume that $k$ is a field of characteristic zero, and that $F$ is a field extension of $k$ or vice-versa. Then, if the noncommutative standard conjecture $C_{NC}$ holds, the category $\NNum(k)_F$ can be made into a Tannakian category $\NNum^\dagger(k)_F$ by modifying its symmetry isomorphism constraints.
\end{theorem}
In order to answer {\bf Question III}, we start by observing that the $F$-linearized Grothendieck group $K_0(A)_F$ of every smooth and proper dg category $\cA$ is endowed with two well-defined equivalence relations: one associated to periodic cyclic homology $(\sim_\mathrm{hom})$ and another one associated to numerical equivalence $(\sim_ \mathrm{num})$; consult \S\ref{sec:homological} for details. This motivates the following conjecture:

\smallbreak

{\bf Noncommutative standard conjecture $D_{NC}(\cA)$:} {\it The following equality~holds
$$ K_0(\cA)_F/\!\!\sim_{\mathrm{hom}} = K_0(\cA)_F/\!\!\sim_{\mathrm{num}}\,.$$}
Its relation with the standard conjecture $D$ (see \S\ref{sec:homological}) is the following.

\begin{theorem}\label{thm:equality1}
Let $k$ and $F$ be fields of characteristic zero with $k$ a field extension of $F$. Then, $D(Z) \Rightarrow D_{NC}(\cD_\perf^\dg(Z))$.
\end{theorem}
Similarly to Theorem~\ref{thm:equality}, Theorem~\ref{thm:equality1} shows us that when restricted to the commutative world, the noncommutative standard conjecture $D_{NC}$ is more likely to hold than the standard conjecture $D$. Hence, it answers affirmatively to the first part of {\bf Question III}. Our answer to the second part is the following:

\begin{theorem}\label{thm:neutral}
Let $k$ be a field of characteristic zero and $F$ a field extension of $k$. If the noncommutative standard conjectures $C_{NC}$ and $D_{NC}$ hold, then $\NNum^\dagger(k)_F$ is neutral Tannakian category. Moreover, an explicit fiber functor neutralizing $\NNum^\dagger(k)_F$ is given by periodic cyclic homology.
\end{theorem}
Theorem~\ref{thm:neutral} combined with the Tannakian formalism gives rise to a Galois group scheme $\mathrm{Gal}(\NNum^\dagger(k)_F)$, which we will call the {\em noncommutative motivic Galois group}. Since by Theorem~\ref{thm:semi} the category $\NNum^\dagger(k)$ is not only abelian but moreover semi-simple, the Galois group $\mathrm{Gal}(\NNum^\dagger(k)_F)$ is {\em pro-reductive}, \ie its unipotent radical is trivial; see \cite[\S2.3.2]{Andre}. Similarly to the super-Tannakian case, the category $\NNum^\dagger(k)_F$ is $\otimes$-equivalent to the category of finite dimensional $F$-valued representations of $\mathrm{Gal}(\NNum^\dagger(k)_F)$.

Finally, making use of all the above results as well as of the Deligne-Milne's theory of Tate triples, we answer {\bf Question IV} as follows:
\begin{theorem}\label{thm:main3}
Let $k$ be a field of characteristic zero. If the standard conjectures $C$ and $D$ as well as the noncommutative standard conjectures $C_{NC}$ and $D_{NC}$ hold, then there exist well-defined surjective group homomorphisms
\begin{equation}\label{eq:surj1}
\mathrm{sGal}(\NNum(k)_k) \twoheadrightarrow \mathrm{Ker}(t: \mathrm{sGal}(\Num(k)_k) \twoheadrightarrow \bbG_m)
\end{equation}
\begin{equation}\label{eq:surj2}
\mathrm{Gal}(\NNum^\dagger(k)_k) \twoheadrightarrow \mathrm{Ker}(t: \mathrm{Gal}(\Num^\dagger(k)_k) \twoheadrightarrow \bbG_m)\,.
\end{equation}
Here, $\bbG_m$ denotes the multiplicative group, $t$ is induced by the inclusion of the category of Tate motives, and the (super-)Galois groups are computed with respect to periodic cyclic homology and de Rham cohomology. 
\end{theorem}
Theorem~\ref{thm:main3} was sketched by Kontsevich in \cite{IAS}. Intuitively speaking, it shows us that the $\otimes$-symmetries of the commutative world which can be lifted to the noncommutative world are precisely those that become trivial when restricted to Tate motives. The difficulty of replacing $k$ by a more general field of coefficients $F$ relies on the lack of an appropriate ``noncommutative \'etale/Betti cohomology''.

In Appendix~\ref{appendix} we collect the main results of the (super-)Tannakian formalism and in Appendix~\ref{app:orbit} (which is of independent interest) we recall the notion of orbit category and describe its behavior with respect to four distinct operations.

\medbreak

\noindent\textbf{Acknowledgments:} The authors are very grateful Eric Friedlander, Dmitry Kaledin and Yuri Manin for useful discussions.

\medbreak

\noindent\textbf{Conventions:} Throughout the article, we will reserve the letter $k$ for the base commutative ring and the letter $F$ for the field of coefficients. The pseudo-abelian envelope construction will be denoted by $(-)^\natural$. 

\section{Differential graded categories}\label{sec:dg}
In this section we collect the notions and results concerning dg categories which are used throughout the article. For further details we invite the reader to consult Keller's ICM address~\cite{ICM}. Let $\cC(k)$ the category of (unbounded) cochain complexes of $k$-modules. A {\em differential graded (=dg) category $\cA$} is a category enriched over $\cC(k)$. Concretely, the morphisms sets $\cA(x,y)$ are complexes of $k$-modules and the composition operation fulfills the Leibniz rule: $d(f\circ g)=d(f)\circ g+(-1)^{\textrm{deg}(f)}f\circ d(g)$. A {\em dg functor} is a functor which preserves the differential graded structure. The category of small dg categories will be denoted by $\dgcat(k)$.

Let $\cA$ be a (fixed) dg category. Its {\em opposite} dg category $\cA^\op$ has the same objects and complexes of morphisms given by $\cA^\op(x,y):=\cA(y,x)$. A {\em right dg $\cA$-module} (or simply an $\cA$-module) is a dg functor $\cA^\op \to \cC_\dg(k)$ with values in the dg category of complexes of $k$-modules. We will denote by $\cC(\cA)$ the category of $\cA$-modules and by $\cD(\cA)$ the {\em derived category} of $\cA$, \ie the localization of $\cC(\cA)$ with respect to the class of quasi-isomorphisms; consult \cite[\S3]{ICM} for details. The full triangulated subcategory of $\cD(\cA)$ formed by its {\em compact objects} \cite[Def.~4.2.7]{Neeman} will be denoted by $\cD_c(\cA)$.

As proved in \cite[Thm.~5.3]{IMRN}, the category $\dgcat(k)$ carries a Quillen model structure whose weak equivalence are the {\em derived Morita equivalences}, \ie the dg functors $\cA \to \cB$ which induce an equivalence $\cD(\cA) \stackrel{\sim}{\to} \cD(\cB)$ on the associated derived categories. The homotopy category hence obtained will be denoted by $\Hmo(k)$. 

The tensor product of $k$-algebras extends naturally to dg categories, giving rise to a symmetric monoidal structure $-\otimes_k-$ on $\dgcat(k)$. The $\otimes$-unit is the dg category $\underline{k}$ with one object and with $k$ as the dg algebra of endomorphisms (concentrated in degree zero). As explained in \cite[\S4.3]{ICM}, the tensor product of dg categories can be naturally derived $-\otimes^\bbL-$ giving thus rise to a symmetric monoidal structure on $\Hmo(k)$. 

Let $\cA$ and $\cB$ be two dg categories. A {\em $\cA\text{-}\cB$-bimodule} is a dg functor $\cA\otimes^\bbL_k \cB^\op \to \cC_\dg(k)$, or in other words a $(\cA^\op \otimes_k^\bbL \cB)$-module. Let $\rep(\cA,\cB)$ be the full triangulated subcategory of $\cD(\cA^\op \otimes_k^\bbL \cB)$ spanned by the (cofibrant) $\cA\text{-}\cB$-bimodules $X$ such that for every object $x \in \cA$ the associated $\cB$-module $X(x,-)$ belongs to $\cD_c(\cB)$.
\subsection{Smooth and proper dg categories}\label{sub:smooth}
Following Kontsevich~\cite{IAS,ENS}, a dg category $\cA$ is called {\em smooth} if the $\cA\text{-}\cA$-bimodule
\begin{eqnarray*}
\cA(-,-): \cA \otimes_k^\bbL \cA^\op \to \cC_\dg(k) && (x,y) \mapsto \cA(x,y)
\end{eqnarray*}
belongs to $\cD_c(\cA^\op \otimes_k^\bbL \cA)$. It is called {\em proper} if for each ordered pair of objects $(x,y)$, the complex $\cA(x,y)$ of $k$-modules belongs to $\cD_c(k)$. As proved in \cite[Thm.~4.8]{CT1}, the smooth and proper dg categories can be conceptually characterized as being precisely the dualizable (or rigid) objects of the symmetric monoidal category $\Hmo(k)$. 
\section{Noncommutative pure motives}
In this section we recall the construction of the categories of noncommutative Chow and numerical motives.
\subsection{Noncommutative Chow motives}\label{sub:Chow}
The rigid symmetric monoidal category $\NChow(k)_F$ of {\em noncommutative Chow motives} (over a base ring $k$ and with coefficients in a field $F$) was proposed by Kontsevich in \cite{IAS} and developed in full detail in \cite{CvsNC, IMRN}. It is defined as the pseudo-abelian envelope of the category:
\begin{itemize}
\item[(i)] whose objects are the smooth and proper dg categories;
\item[(ii)]  whose morphisms from $\cA$ to $\cB$ are given by the $F$-linearized Grothendieck group $K_0(\cA^\op \otimes_k^\bbL \cB)_F$;
\item[(iii)] whose composition law is induced by the (derived) tensor product of bimodules.
\end{itemize}
Its symmetric monoidal structure is induced by the (derived) tensor product of dg categories. In analogy with the commutative world, the morphisms of $\NChow(k)_F$ will be called {\em noncommutative correspondences}. By definition a noncommutative Chow motive consists then of a pair $(\cA,e)$, where $\cA$ a smooth and proper dg category and $e$ an idempotent of the $F$-algebra $K_0(\cA^\op \otimes_k^\bbL \cA)_F$. When $e=[\cA(-,-)]$, we will simply write $\cA$ instead of $(\cA,[\cA(-,-)])$. 
\subsection{Noncommutative numerical motives}\label{sub:Num}
The rigid symmetric monoidal category $\NNum(k)_F$ of {\em noncommutative numerical motives} (over a base ring $k$ and with coefficients in a field $F$) was constructed\footnote{Kontsevich had previously introduced a category $\mathrm{NC}_{\mathrm{num}}(k)_F$ of noncommutative numerical motives; see~\cite{IAS}. However, the authors have recently proved that $\mathrm{NC}_{\mathrm{num}}(k)_F$ and $\NNum(k)_F$ are in fact $\otimes$-equivalent; see~\cite{Konts}.} by the authors in \cite{Semi}. Let $(\cA,e)$ and $(\cB,e')$ be two noncommutative Chow motives and $\underline{X}=(e \circ [\sum_i a_iX_i]\circ e')$ and $\underline{Y}=(e' \circ [\sum_j b_jY_j]\circ e)$ 
two noncommutative correspondences. Recall that $X_i$ is a $\cA\text{-}\cB$-bimodule, that $Y_j$ is a $\cB\text{-}\cA$-bimodule, and that the sums are indexed by a finite set. The {\em intersection number} $ \langle \underline{X} \cdot \underline{Y} \rangle$ of $\underline{X}$ with $\underline{Y}$ is given by the following formula
\begin{equation*}
\sum_{i,j} a_i \cdot b_j \cdot [HH(\cA;X_i \otimes^\bbL_\cB Y_j)] \in K_0(k)_F\,,
\end{equation*}
where $[HH(\cA;X_i\otimes^\bbL_\cB Y_j)]$ denotes the class in $K_0(k)_F$ of the Hochschild homology complex of $\cA$ with coefficients in the $\cA\text{-}\cA$-bimodule $X_i \otimes^\bbL_\cB Y_j$. A noncommutative correspondence $\underline{X}$ is called {\em numerically equivalent to zero} if for every noncommutative correspondence $\underline{Y}$ the intersection number $\langle \underline{X} \cdot \underline{Y}\rangle$ is zero. As proved in \cite[Thm.~1.5]{Semi}, the noncommutative correspondences which are numerically equivalent to zero form a $\otimes$-ideal $\cN$ of the category $\NChow(k)_F$. Moreover, $\cN$ is the largest $\otimes$-ideal of $\NChow(k)_F$ (distinct from the entire category). The category of noncommutative numerical motives $\NNum(k)_F$ is then by definition the pseudo-abelian envelope of the quotient category $\NChow(k)_F/\cN$.

\section{Change of coefficients}\label{sec:coefficients}
In this section, we explore the change-of-coefficients mechanism. This will allows us to improve the main results of \cite{Semi} concerning the semi-simplicity of the category $\NNum(k)_F$ and the relation between the commutative and the noncommutative world; see Proposition~\ref{prop:extension} and Theorem~\ref{thm:semi}. In what follows, $F$ will be a field, $K$ a field extension of $F$, and $(\cC,\otimes, {\bf 1})$ a $F$-linear, additive, rigid symmetric monoidal category such that $\End_\cC({\bf 1})\simeq F$. This data allows us to consider a new category $\cC\otimes_FK$. It has the same objects as $\cC$ and morphisms given by
$$ \Hom_{\cC\otimes_FK}(X,Y):= \Hom_\cC(X,Y) \otimes_F K\,.$$
By construction, it is $K$-linear, additive, rigid symmetric monoidal, and such that $\End_{\cC\otimes_FK}({\bf 1})\simeq F\otimes_FK \simeq K$. Moreover, it comes equipped with a canonical $F$-linear symmetric monoidal functor $\cC \to \cC\otimes_F K$. 
\subsection*{Quotient categories}
Recall from \cite[Lemma~7.1.1]{AK} that the formula
$$ \cN_\cC(X,Y):= \{ f \in \Hom_\cC(X,Y) \, |\, \forall g \in \Hom_\cC(Y,X), \, \mathrm{tr}(g\circ f)=0\}$$
defines a $\otimes$-ideal $\cN_\cC$ of $\cC$, where $\mathrm{tr}$ stands for the categorical trace. Moreover, $\cN_\cC$ can be characterized as the largest $\otimes$-ideal of $\cC$ (distinct from the entire category); see \cite[Prop.~7.1.4]{AK}. As proved in \cite[Prop.~1.4.1]{Bruguieres}, the change-of-coefficients mechanism is well-behaved with respect to this $\otimes$-ideal, \ie the canonical functor
\begin{equation}\label{eq:Brug}
(\cC/\cN_\cC) \otimes_F K \stackrel{\sim}{\too} (\cC \otimes_F K)/\cN_{\cC\otimes_FK}
\end{equation}
is an additive $\otimes$-equivalence of categories.
\subsection*{Motives versus noncommutative motives}
As explained in \S\ref{sub:Chow}, the category $\NChow(k)_F$ of noncommutative Chow motives is $F$-linear, additive, and rigid symmetric monoidal. When $k$ is furthermore local (or more generally when $K_0(k)=\bbZ$), we have
$$ \End_{\NChow(k)_F}(\underline{k})\simeq K_0(\underline{k}^\op \otimes_k^\bbL \underline{k})_F \simeq K_0(k)_F \simeq F\,.$$
By construction we then observe that the canonical functor
\begin{equation}\label{eq:Chow}
(\NChow(k)_F \otimes_FK)^\natural \stackrel{\sim}{\too} \NChow(k)_K
\end{equation}
is an additive $\otimes$-equivalence of categories. By combining \eqref{eq:Brug} with \eqref{eq:Chow}, we then obtain the following additive $\otimes$-equivalence of categories
\begin{equation}\label{eq:Numscalars}
(\NNum(k)_F \otimes_FK)^\natural \stackrel{\sim}{\too} \NNum(k)_K\,.
\end{equation}
\begin{proposition}\label{prop:extension}
Let $k$ and $F$ be fields with $F$ of characteristic zero. Then, by applying the change-of-coefficients construction $-\otimes_\bbQ F$ to \eqref{eq:diagram} we obtain the following diagram of $F$-linear, additive, symmetric monoidal functors
\begin{equation}\label{eq:diagram-new}
\xymatrix@C=2em@R=1.5em{
& \Chow(k)_F  \ar[d]_-{\tau} \ar[dl] &&  \\
\Num(k)_F \ar[d]_-{\tau} & \Chow(k)_F/_{\!\!-\otimes \bbQ(1)}  \ar[dl]  \ar[rr]^-R && \ar[dl] \NChow(k)_F  \\
\Num(k)_F/_{\!\!-\otimes \bbQ(1)} \ar[rr]_{R_{\cN}} & & \NNum(k)_F & \,,
}
\end{equation}
where the functors $\tau$ are faithful and the functors $R$ and $R_\cN$ fully-faithful.
\end{proposition}
\begin{proof}
As explained in \cite[\S4.2.2]{Andre} we have $\otimes$-equivalences
\begin{eqnarray*}
(\Chow(k)_\bbQ \otimes_\bbQ F)^\natural \stackrel{\sim}{\too} \Chow(k)_F && (\Num(k)_\bbQ \otimes_\bbQ F)^\natural \stackrel{\sim}{\too} \Num(k)_F\,.
\end{eqnarray*}
By combining Lemmas~\ref{lem:orbit} and \ref{lem:orbit-pseudo} we obtain induced fully-faithful functors
$$ (\Chow(k)_\bbQ \otimes_\bbQ F)^\natural \too (\Chow(k)_\bbQ\otimes_\bbQ F)^\natural /_{\!\!-\otimes \bbQ(1)} \too ((\Chow(k)_\bbQ/_{\!\!-\otimes \bbQ(1)})\otimes_\bbQ F)^\natural$$
$$ (\Num(k)_\bbQ \otimes_\bbQ F)^\natural \too (\Num(k)_\bbQ\otimes_\bbQ F)^\natural /_{\!\!-\otimes \bbQ(1)} \too ((\Num(k)_\bbQ/_{\!\!-\otimes \bbQ(1)})\otimes_\bbQ F)^\natural\,.$$
Hence, the proof follows from the above $\otimes$-equivalences \eqref{eq:Chow} and \eqref{eq:Numscalars}.
\end{proof}
\subsection*{Semi-simplicity}
In this subsection we switch the role of $F$ and $K$, \ie we assume that $F$ is a field extension of $K$.
\begin{theorem}\label{thm:semi}
Assume one of the following two conditions:
\begin{itemize}
\item[(i)] The base ring $k$  is local (or more generally we have $K_0(k)=\bbZ$) and $F$ is a $k$-algebra; a large class of examples is given by taking $k=\bbZ$ and $F$ an arbitrary field.
\item[(ii)] The base ring $k$ and the field $F$ are two field extensions of a (non-trivial) field $K$; a large classe of examples is given by taking $K=F=\bbQ$ and $k$ a field of characteristic zero, or $K=\bbQ$, $F=\overline{\bbQ}$, and $k$ a field of characteristic zero.
\end{itemize}
Then, the category $\NNum(k)_F$ is abelian semi-simple. Moreover, if $\cJ$ is a $\otimes$-ideal in $\NChow(k)_F$ for which the pseudo-abelian envelope of the quotient category $\NChow(k)_F/\cJ$ is abelian semi-simple, then $\cJ$ agrees with $\cN$.
\end{theorem}
\begin{remark} Theorem~\ref{thm:semi} extends the original semi-simplicity result~\cite[Thm.~1.9]{Semi}. The latter is obtained from the former by taking $K=F$.
\end{remark}
\begin{proof}
Condition (i) is the same as the one of \cite[Thm.~1.9]{Semi}. Hence, let us assume condition (ii). Since $k$ is a field extension of $K$ we conclude by \cite[Thm.~1.9]{Semi} that the category $\NNum(k)_K$ is abelian semi-simple. By hypothesis, $F$ is also a field extension of $K$ and so by \eqref{eq:Numscalars} (with $F$ and $K$ switched) we obtain an equivalence of categories
$$ (\NNum(k)_K \otimes_K F)^\natural \stackrel{\sim}{\too} \NNum(k)_F\,.$$
In order to prove that $\NNum(k)_F$ is abelian semi-simple it suffices then to show that for every object $N \in \NNum(k)K \otimes_K F$, the $F$-algebra $\End_{\NNum(k)_K \otimes_KF}(N)$ of endomorphisms is finite dimensional and semi-simple; see \cite[Lemma~2]{Jannsen}. Note that we have a natural isomorphism
\begin{equation}\label{eq:equality}
\End_{\NNum(k)_K \otimes_KF}(N) \simeq \End_{\NNum(k)_K}(N)\otimes_K F\,.
\end{equation}
The category $\NNum(k)_K$ is abelian semi-simple and so the $K$-algebra of endomorphisms $\End_{\NNum(k)_K}(N)$ is finite dimensional and its Jacobson radical is trivial. By \eqref{eq:equality}, we then conclude that the $F$-algebra $\End_{\NNum(k)_K \otimes_KF}(N)$ is also finite dimensional. Moreover, since the Jacobson radical of $\End_{\NNum(k)_K \otimes_KF}(N)$ is obtained from the one of $\End_{\NNum(k)_K}(N)$ by base change we conclude that it is also trivial; see \cite[Prop.~4.1.1]{AK}. This implies that the $F$-algebra $\End_{\NNum(k)_K\otimes_KF}(N)$ is semi-simple and so the category $\NNum(k)_F$ is abelian semi-simple. The second claim of the theorem, concerning the $\otimes$-ideal $\cJ$, follows from \cite[Prop.~7.1.4 c)]{AK}.
\end{proof}

\section{Schur-finiteness}\label{sec:Schur}
In what follows, $F$ will be a field of characteristic zero. Let $(\cC, \otimes, {\bf 1})$ be a $F$-linear, idempotent complete, symmetric monoidal category. As explained in \cite[\S3.1]{Andre-BourbakiSem}, given an object $X \in \cC$ and an integer $n \geq 1$, the symmetric group $S_n$ acts on  $X^{\otimes n}$ by permutation of its factors. The isomorphism classes $V_\lambda$ of the irreducible $\bbQ$-linear representations of $S_n$ are in canonical bijection with the partitions $\lambda$ of $n$. As a consequence, the group ring $\bbQ[S_n]$ can be written as $\prod_{\lambda, |\lambda|=n} \mathrm{End}_\bbQ V_\lambda$. Let $c_\lambda$ be the idempotent of $\bbQ[S_n]$ defining the representation $V_\lambda$. The {\em Schur functor} $S_\lambda$ is then defined by the following formula
\begin{eqnarray*}
S_\lambda: \cC \too \cC && X \mapsto S_\lambda(X):= c_\lambda(X^{\otimes n})\,.
\end{eqnarray*}
Note that for $\lambda=(n)$ the Schur functor $S_{(n)}$ is the
symmetric power functor, while for $\lambda=(1,\ldots, 1)$ the Schur functor $S_{(1,\ldots,1)}$ 
is the exterior power functor. An object $X \in \cC$ is called {\em Schur-finite} if there exists an integer $n \geq 1$ and a partition $\lambda$ of $n$ such that $X$ is annihilated by the Schur functor of $\lambda$, \ie $S_\lambda(X)=0$. The category $\cC$ is called {\em Schur-finite} if all its objects are Schur-finite.
\begin{lemma}\label{lem:Schur}
Let $L:\cC_1 \to \cC_2$ be a $F$-linear symmetric monoidal functor.
\begin{itemize}
\item[(i)] If $X \in \cC_1$ is Schur-finite, then $L(X)$ is Schur-finite.
\item[(ii)] When $L$ is moreover faithful, then the converse holds. In other words, if $L(X)$ is Schur-finite, then $X$ is Schur-finite.
\end{itemize}
\end{lemma}
The proof of Lemma~\ref{lem:Schur} is an easy exercise that we leave for the reader. 

\section{Super-Tannakian structure}\label{sec:super}
In this section we prove Theorem~\ref{thm:main1}.

\begin{proposition}\label{prop:Schur}
Let $k$ and $F$ be as in Theorem~\ref{thm:semi} with $F$ of characteristic zero. Then, the category $\NNum(k)_F$ is Schur-finite.
\end{proposition}
\begin{proof}
Note first that by construction the category $\NNum(k)_F$ is $F$-linear, idempotent complete, and symmetric monoidal. Let us assume first that $k$ and $F$ satisfy condition (i) of Theorem~\ref{thm:semi}. Then, as explained in the proof of \cite[Thm.~1.9]{Semi}, Hochschild homology $(HH)$ gives rise to a $F$-linear symmetric monoidal functor
\begin{equation}\label{eq:HH}
\overline{HH}: \NChow(k)_F \too \cD_c(F)\,.
\end{equation}
Let us denote by $\mathrm{Ker}(\overline{HH})$ the associated kernel. Since this is a $\otimes$-ideal of $\NChow(k)_F$ the induced functor
$$ (\NChow(k)_F/\mathrm{Ker}(\overline{HH}))^\natural \too \cD_c(F)$$
is not only $F$-linear and faithful but moreover symmetric monoidal. Note that the category $\cD_c(F)$ can be naturally identified with the category of those $\bbZ$-graded $F$-vector spaces $\{ V_n \}_{n \in \bbZ}$ such that $\mathrm{dim} (\oplus_n V_n) < \infty$. As a consequence, $\cD_c(F)$ is Schur-finite. By Lemma~\ref{lem:Schur}(ii) we then conclude that $(\NChow(k)_F/\mathrm{Ker}(\overline{HH}))^\natural$ is Schur-finite. Now, recall from \S\ref{sub:Num} that the ideal $\cN$ is the largest $\otimes$-ideal of $\NChow(k)_F$ (distinct from the entire category). The inclusion $\mathrm{Ker}(\overline{HH})\subset \cN$ of $\otimes$-ideals gives then rise to a $F$-linear symmetric monoidal functor
\begin{equation}\label{eq:func-idemp}
(\NChow(k)_F/ \mathrm{Ker}(\overline{HH}))^\natural \too (\NChow(k)_F/\cN)^\natural =: \NNum(k)_F\,.
\end{equation}
By combining Lemma~\ref{lem:Schur}(i) with the fact that the category $(\NChow(k)_F/\mathrm{Ker}(\overline{HH}))^\natural$ is Schur-finite, we then conclude that all noncommutative numerical motives in the image of the functor \eqref{eq:func-idemp} are Schur-finite. Finally, since every noncommutative numerical motive is a direct factor of one of these and Schur-finiteness is clearly stable under direct factors, we then conclude that the category $\NNum(k)_F$ is Schur-finite.

Let us now assume that $k$ and $F$ satisfy condition (ii) of Theorem~\ref{thm:semi}. If $K=F$, then $k$ is a field extension of $F$ and so the proof is similar to the one above: the functor \eqref{eq:HH} take values not in $\cD_c(F)$ but in $\cD_c(k)$, but since $\cD_c(k)$ is also Schur-finite the same reasoning applies. Now, let us assume that $k$ and $F$ are two field extension of a (non-trivial) field  $K$. Note that since $F$ is of characteristic zero, $K$ is also of characteristic zero. In this case, $k$ is a field extension of $K$, and so the preceding arguments shows us that $\NNum(k)_K$ is Schur-finite. As explained in \S\ref{sec:coefficients} we have a canonical $K$-linear symmetric monoidal functor $\NNum(k)_K \to \NNum(k)_K \otimes_K F$. By construction the categories $\NNum(k)_K$ and $\NNum(k)_K \otimes_K F$ have the same objects and so Lemma~\ref{lem:Schur}(i), combined with the fact that Schur-finiteness is stable under direct factors, implies that $(\Num(k)_K\otimes_K F)^\natural$ is also Schur-finite. Finally, the $\otimes$-equivalence \eqref{eq:Numscalars} (with $F$ and $K$ switched) allows us to conclude that $\NNum(k)_F$ is Schur-finite.
\end{proof}
Proposition~\ref{prop:Schur} implies a similar result in the commutative world.
\begin{corollary}
Let $k$ and $F$ be a fields of characteristic zero. Then, the category $\Num(k)_F$ is Schur-finite.
\end{corollary}
\begin{proof}
Recall from diagram \eqref{eq:diagram-new} the following composition of faithful, $F$-linear, additive, symmetric monoidal functors 
\begin{equation*}
\Num(k)_F \too \Num(k)_F/_{\!\!-\otimes \bbQ(1)} \stackrel{R_\cN}{\too} \NNum(k)_F\,.
\end{equation*}
Since $k$ and $F$ are of characteristic zero, they satisfy condition (ii) of Theorem~\ref{thm:semi} (with $K=\bbQ$). Then, Proposition~\ref{prop:Schur} combined with Lemma~\ref{lem:Schur}(ii) allows us to conclude that the category $\Num(k)_F$ is Schur-finite. 
\end{proof}
\subsection*{Proof of Theorem~\ref{thm:main1}}
Note first that by construction the category $\NNum(k)_F$ is $F$-linear, additive, and rigid symmetric monoidal. Its $\otimes$-unit is the noncommutative Chow motive $\underline{k}$. Since $k$ and $F$ are as in Theorem~\ref{thm:semi}, the equality $$\End_{\NNum(k)_F}(\underline{k})\simeq K_0(k)_F\simeq F$$ holds. Moreover, due to Theorem~\ref{thm:semi} the category $\NNum(k)_F$ is also abelian (semi-simple). By Deligne's intrinsic characterization (see Theorem~\ref{sutannDel}) it suffices then to show that $\NNum(k)_F$ is Schur-finite. This is the content of Proposition~\ref{prop:Schur} and so the proof is finished.
\section{Periodic cyclic homology}\label{sec:HP}
In this section we prove that periodic cyclic homology ($HP$) gives rise to a well-defined symmetric monoidal functor on the category of noncommutative Chow motives; see Theorem~\ref{thm:HP}. In what follows, $k$ will be a field.

Following Kassel~\cite[\S1]{Kassel}, a {\em mixed complex} $(M,b,B)$ is a $\bbZ$-graded $k$-vector space $\{M_n\}_{n \in \bbZ}$ endowed with a degree $+1$ endomorphism $b$ and a degree $-1$ endomorphism $B$ satisfying the relations $b^2=B^2=Bb +bB=0$. Equivalently, a mixed complex is a right dg module over the dg algebra $\Lambda:=k[\epsilon]/\epsilon^2$, where $\epsilon$ is of degree $-1$ and $d(\epsilon)=0$. Recall from \cite[Example~7.10]{CT1} the construction of the mixed complex functor $C: \dgcat(k) \to \cD(\Lambda)$ with values in the derived category of $\Lambda$. As explained by Kassel in~\cite[page~210]{Kassel}, there is a well-defined {\em $2$-perioditization functor} sending a mixed complex $(M,b,B)$ to the following $\bbZ/2$-graded complex of $k$-vector spaces
$$
 \prod_{\mathrm{n\,even}} \xymatrix{M_n \ar@<1ex>[r]^-{b+B} & \ar@<1ex>[l]^-{b+B}}  \prod_{\mathrm{n\,odd}} M_n \,.
$$
This functor preserves weak equivalences and when combined with $C$ gives rise to periodic cyclic homology
\begin{equation}\label{eq:HP1}
HP: \dgcat(k) \stackrel{C}{\too} \cD(\Lambda) \too \cD_{\bbZ/2}(k)\,.
\end{equation}
Here, $\cD_{\bbZ/2}(k)$ stands for the derived category of $\bbZ/2$-graded complexes.
\begin{theorem}\label{thm:HP} When $F$ is a field extension of $k$, the functor \eqref{eq:HP1} gives rise to a $F$-linear symmetric monoidal functor
\begin{equation}\label{eq:induced}
\overline{HP_\ast}: \NChow(k)_F \too \mathrm{sVect}(F)
\end{equation}
with values in the category of finite dimensional super $F$-vector spaces. On the other hand, when $k$ is a field extension of $F$, the functor \eqref{eq:HP1} gives rise to a $F$-linear symmetric monoidal functor
\begin{equation}\label{eq:induced1}
\overline{HP_\ast}: \NChow(k)_F \too \mathrm{sVect}(k)
\end{equation}
with values in the category of finite dimensional super $k$-vector spaces.
\end{theorem}

\begin{proof}
As explained in \cite[Example~7.10]{CT1}, the mixed complex functor $C$ is symmetric monoidal. On the contrary, the $2$-perioditization functor is not symmetric monoidal. This is due to the fact that it uses infinite products and these do not commute with the tensor product. Nevertheless, as explained in \cite[page~210]{Kassel}, the $2$-perioditization functor is lax symmetric monoidal. Note that since by hypothesis $k$ is a field, the category $\cD_{\bbZ/2}(k)$ is $\otimes$-equivalent to the category $\mathrm{SVect}(k)$ of super $k$-vector spaces. Hence, the functor \eqref{eq:HP1} gives rise to a well-defined lax symmetric monoidal functor
\begin{equation}\label{eq:auxiliar1}
HP_\ast: \dgcat(k) \stackrel{C}{\too} \cD(\Lambda) \too \cD_{\bbZ/2}(k)\simeq \mathrm{SVect}(k)\,.
\end{equation} 
Now, recall from \cite[\S5]{IMRN} the construction of the additive category $\Hmo_0(k)$. Its objects are the small dg categories, its morphisms from $\cA$ to $\cB$ are given by the Grothendieck group $K_0\rep(\cA,\cB)$ of the triangulated category $\rep(\cA,\cB)$ (see \S\ref{sec:dg}), and its composition law is induced by the (derived) tensor product of bimodules. The (derived) tensor product of dg categories endows $\Hmo_0(k)$ with a symmetric monoidal structure. Moreover, there is a natural symmetric monoidal functor $\cU: \dgcat(k) \to \Hmo_0(k)$ which can be characterized as the {\em universal additive invariant}; consult \cite{Survey,IMRN} \cite[\S5.1]{ICM} for details. The above functor \eqref{eq:auxiliar1} is an example of an {\em additive invariant}, \ie it inverts derived Morita equivalences (see \S\ref{sec:dg}) and maps semi-orthogonal decompositions in the sense of Bondal-Orlov \cite{BO1} into direct sums. Hence, by the universal property of $\cU$, the additive invariant \eqref{eq:auxiliar1} gives rise to a well-defined lax symmetric monoidal functor
\begin{equation}\label{eq:auxiliar2}
\overline{HP_\ast}: \Hmo_0(k) \too \mathrm{SVect}(k)
\end{equation}
such that $\overline{HP_\ast}\circ \cU = HP_\ast$. Now, let us denote by $\Hmo_0(k)^{\mathrm{sp}} \subset \Hmo_0(k)$ the full subcategory of smooth and proper dg categories in the sense of Kontsevich; see \S\ref{sub:smooth}. Given smooth and proper dg categories $\cA$ and $\cB$, there is a natural equivalence of categories $\rep(\cA,\cB)\simeq \cD_c(\cA^\op \otimes_k \cB)$. As a consequence, we have the following description of the Hom-sets
$$ \Hom_{\Hmo_0(k)^{\mathsf{sp}}}(\cA,\cB):=K_0\rep(\cA,\cB) \simeq K_0(\cA^\op \otimes_k \cB)\,.$$
This allows us to conclude that the category $\NChow(k)_F$ can be obtained from $\Hmo_0(k)^{\mathrm{sp}}$ by first tensoring each abelian group of morphisms with the field $F$ and then passing to the associated pseudo-abelian envelope. Schematically, we have the following composition
\begin{equation}\label{eq:auxiliar3}
\Hmo_0(k)^{\mathrm{sp}} \stackrel{(-)_F}{\too} \Hmo_0(k)_F^{\mathrm{sp}} \stackrel{(-)^\natural}{\too} \NChow(k)_F\,.
\end{equation} 
By Proposition~\ref{prop:aux}(ii), the restriction of \eqref{eq:auxiliar2} to $\Hmo_0(k)^{\mathrm{sp}}$ gives then rise to a lax symmetric monoidal functor
\begin{equation}\label{eq:auxiliar4}
\overline{HP_\ast}: \Hmo_0(k)^{\mathrm{sp}} \too \mathrm{sVect}(k)\,.
\end{equation}
Moreover, conditions (i) and (ii) of Proposition~\ref{prop:aux} are precisely what Emmanouil named ``property (II)'' in \cite[page~211]{Emmanouil}. As a consequence, \cite[Thm.~4.2]{Emmanouil} implies that the functor \eqref{eq:auxiliar4} is in fact symmetric monoidal. 

We now have all the ingredients needed for the description of the functors \eqref{eq:induced} and \eqref{eq:induced1}. Let us assume first that $F$ is a field extension of $k$. Then, we have an induced extension of scalars functor $\mathrm{sVect}(k) \to \mathrm{sVect}(F)$. This functor is symmetric monoidal and the category $\mathrm{sVect}(F)$ is clearly idempotent complete. Using the description of $\NChow(k)_F$ given at diagram \eqref{eq:auxiliar3}, we then conclude that the above functor \eqref{eq:auxiliar4} extends to $\NChow(k)_F$, thus giving rise to the $F$-linear symmetric monoidal functor \eqref{eq:induced}. Now, let us assume that $k$ is a field extension of $F$. In this case the category $\mathrm{sVect}(k)$ is already $F$-linear. Hence, using the description of $\NChow(k)_F$ given at diagram \eqref{eq:auxiliar3}, we conclude that the above functor \eqref{eq:auxiliar4} gives rise to the induced $F$-linear symmetric monoidal functor \eqref{eq:induced1}. The proof is then finished.
\end{proof}

\begin{proposition}\label{prop:aux}
Let $\cA$ be a smooth and proper dg category. Then, the following two conditions hold:
\begin{itemize}
\item[(i)] The inverse system $(HC(\cA)[-2m],S)_m$ of cyclic homology $k$-vector spaces (see \cite[\S2.2]{Loday}) satisfies the Mittag-Leffler condition;
\item[(ii)] The periodic cyclic homology $k$-vector spaces $HP_n(\cA)$ are finite dimensional.
\end{itemize}
\end{proposition}
\begin{proof}
Let us start by proving the following two conditions:
\begin{itemize}
\item[(a)] the Hochschild homology $k$-vector spaces $HH_n(\cA)$ vanish for $|n|\gg 0$;
\item[(b)] the Hochschild homology $k$-vector spaces $HH_n(\cA)$ are finite dimensional. 
\end{itemize}
Recall from \cite[Example~7.9]{CT1} that the functor $HH: \dgcat(k) \to \cD(k)$ is symmetric monoidal. Moreover, it inverts derived Morita equivalences and so it descends to the homotopy category $\Hmo(k)$. As explained in \S\ref{sub:smooth}, every smooth and proper dg category $\cA$ is dualizable in the symmetric monoidal category $\Hmo(k)$. As a consequence, $HH(\cA)$ is a dualizable object in $\cD(k)$, \ie it belongs to $\cD_c(k)$. This implies that conditions (a) and (b) are verified. Now, condition (a) combined with Connes' periodicity exact sequence (see \cite[Thm.~2.2.1]{Loday})
\begin{equation}\label{eq:SBI}
\cdots \stackrel{B}{\too} HH_n(\cA) \stackrel{I}{\too} HC_n(\cA) \stackrel{S}{\too} HC_{n-2}(\cA) \stackrel{B}{\too} HH_{n-1}(\cA) \stackrel{I}{\too} \cdots\,,
\end{equation}
allows us to conclude that the map $S$ is an isomorphism for $|n|\gg 0$. Hence, the inverse system
$$ \cdots \stackrel{S}{\too} HC_{n+2r}(\cA) \stackrel{S}{\too} HC_{n+2r-2}(\cA) \stackrel{S}{\too} \cdots \stackrel{S}{\too} HC_n(\cA)$$
satisfies the Mittag-Leffler condition, and so condition (i) is verified. As a consequence we obtain the following equality
\begin{equation}\label{eq:description}
HP_n(\cA)= \underset{r}{\mathrm{lim}}\, HC_{n+2r}(\cA)\,.
\end{equation}
Condition (b) and the long exact sequence \eqref{eq:SBI} imply that $HC_m(\cA)$ is a finite dimensional $k$-vector space for every $m \in \bbZ$. This fact combined with equality \eqref{eq:description} and with the fact that $S$ is an isomorphism for $|n|\gg 0$ allows us then to conclude that $HP_n(\cA)$ is a finite dimensional $k$-vector space. Condition (ii) is then verified and so the proof is finished.
\end{proof}

\section{Noncommutative standard conjecture $C_{NC}$}\label{sec:Kunneth}
In this section we start by recalling the standard conjecture $C$ and the sign conjecture $C^+$. Then, we prove that the noncommutative standard conjecture $C_{NC}$ is stable under tensor products (see Proposition~\ref{prop:stable-tensor}) and finally Theorem~\ref{thm:equality}. In what follows, $Z$ will be a smooth projective $k$-scheme over a field $k$ of characteristic zero.

Recall from \cite[\S3.4]{Andre} that with respect to de Rham cohomology, we have the following K{\"u}nneth projectors
\begin{eqnarray*}
\pi^n_{Z}: H_{dR}^\ast(Z) \twoheadrightarrow  H^n_{dR}(Z) \hookrightarrow H_{dR}^\ast(Z) && 0 \leq n \leq 2\, \mathrm{dim}(Z)\,.
\end{eqnarray*}

\smallbreak

{\bf Standard conjecture $C(Z)$ (see \cite[\S5.1.1.1]{Andre})}: {\it The K\"unneth projectors $\pi^n_Z$ are algebraic, \ie they can be written as  $\pi^n_{Z}=H^\ast_{dR}(\underline{\pi}^n_{Z})$, with $\underline{\pi}^n_{Z}$ algebraic correspondences.}

{\bf Sign conjecture $C^+(Z)$ (see \cite[\S5.1.3]{Andre})}: {\it The K{\"u}nneth projectors $\pi^+_{Z}:= \sum_{n=0}^{\mathrm{dim}(Z)} \pi^{2n}_{Z}$ (and hence $\pi^-_{Z}:= \sum_{n=1}^{\mathrm{dim}(Z)} \pi^{2n-1}_{Z}$)
are algebraic, \ie they can be written as $\pi^+_{Z}=H^\ast_{dR}(\underline{\pi}^+_{Z})$, with $\underline{\pi}^+_{Z}$ algebraic correspondences.
}
\begin{remark}
Clearly $C(Z) \Rightarrow C^+(Z)$. Recall also from \cite[\S5.1.1.2]{Andre} that $C(Z)$ and $C^+(Z)$ could equivalently be formulated using any other classical Weil cohomology.
\end{remark}

\begin{proposition}\label{prop:stable-tensor}
Let $\cA$ and $\cB$ be two smooth and proper dg categories; see \S\ref{sub:smooth}. Then, the following implication holds
\begin{equation}\label{eq:implication}
C_{NC}(\cA) + C_{NC}(\cB) \Rightarrow C_{NC} (\cA \otimes_k \cB)\,.
\end{equation}
\end{proposition}
\begin{proof}
Recall from Theorem~\ref{thm:HP} that we have a well-defined $F$-linear symmetric monoidal functor
\begin{equation}\label{eq:func-mon}
\overline{HP_\ast}: \NChow(k)_F \too \mathrm{sVect}(K)\,.
\end{equation}
The field $K$ is equal to $F$ when $F$ is a field extension of $k$ and is equal to $k$ when $k$ is a field extension of $F$. Let us denote by $-\widehat{\otimes}-$ the symmetric monoidal structure on $\mathrm{sVect}(K)$. Then, we have the following equalities between the K\"unneth projectors
\begin{eqnarray*}
\pi^+_{\cA \otimes_k \cB} = \pi^+_\cA \widehat{\otimes} \pi^+_\cB + \pi^+_\cA \widehat{\otimes} \pi^-_\cB &&
\pi^-_{\cA \otimes_k \cB} = \pi^+_\cA \widehat{\otimes} \pi^-_\cB + \pi^-_\cA \widehat{\otimes} \pi^-_\cB\,.
\end{eqnarray*}
Since \eqref{eq:func-mon} is symmetric monoidal and by hypothesis we have noncommutative correspondences $\underline{\pi}^\pm_\cA$ and $\underline{\pi}^\pm_\cB$ such that $\pi^\pm_\cA= \overline{HP_\ast}(\underline{\pi}^\pm_\cA)$ and $\pi^\pm_\cB= \overline{HP_\ast}(\underline{\pi}^\pm_\cB)$, we conclude that the noncommutative standard conjecture $C_{NC}(\cA\otimes_k \cB)$ also holds. As a consequence, the above implication \eqref{eq:implication} holds and so the proof is finished.
\end{proof}
\subsection*{Proof of Theorem~\ref{thm:equality}}
As explained in \cite[\S4.2.5]{Andre}, de Rham cohomology $H^\ast_{dR}$ (considered as a Weil cohomology) gives rise to a symmetric monoidal functor
\begin{equation}\label{eq:deRham}
\overline{H^\ast_{dR}}: \Chow(k)_F \too \mathrm{GrVect}(k)
\end{equation}
with values in the category of finite dimensional $\bbZ$-graded $k$-vector spaces. On the other hand, as explained in Theorem~\ref{thm:HP}, periodic cyclic homology gives rise to a well-defined symmetric monoidal functor
\begin{equation}\label{eq:Hperiodic}
\overline{HP_\ast}: \Chow(k)_F \too \mathrm{sVect}(k)\,.
\end{equation}
Recall from diagram \eqref{eq:diagram-new} the following sequence of functors
\begin{equation}\label{eq:compChow}
\Chow(k)_F \stackrel{\tau}{\too} \Chow(k)_F/_{\!\!-\otimes \bbQ(1)} \stackrel{R}{\too} \NChow(k)_F\,.
\end{equation}
By combining \cite[Thm.~1.1]{CvsNC} with the change-of-coefficients mechanism we conclude that the image of $Z$ under the above composition \eqref{eq:compChow} identifies naturally with the noncommutative Chow motive $\cD_\perf^\dg(Z)$. Now, recall from Keller~\cite{Ringed} that $\overline{HP_\ast}(\cD_\perf^\dg(Z))\simeq HP_\ast(\cD_\perf^\dg(Z))$ agrees with the periodic cyclic homology $HP_\ast(Z)$ of the $k$-scheme $Z$ in the sense of Weibel~\cite{Weibel}. Since $k$ is a field of characteristic zero and $Z$ is smooth, the Hochschild-Konstant-Rosenberg theorem~\cite{Hodge} furnish us the following isomorphisms
\begin{eqnarray*}
HP^+_\ast(Z) \simeq \bigoplus_{n\, \even}H_{dR}^n(Z)  && HP^-_\ast(Z) \simeq \bigoplus_{n\, \odd}H_{dR}^n(Z) \,.
\end{eqnarray*}
These facts allows us to conclude that the composition of \eqref{eq:Hperiodic} with \eqref{eq:compChow} is the following functor
\begin{eqnarray}\label{eq:sHdR}
& s\overline{H^\ast_{dR}}: \Chow(k)_F \too \mathrm{sVect}(k) & Z \mapsto \left(\bigoplus_{n\, \even}H_{dR}^n(Z),\bigoplus_{n\, \odd}H_{dR}^n(Z)\right) 
\end{eqnarray}
naturally associated to \eqref{eq:deRham}. As a consequence, the K\"unneth projectors $\pi^\pm_Z$ of the sign conjecture $C^+(Z)$ agree with the K\"unneth projectors $\pi^\pm_{\cD_\perf^\dg(Z)}$ of the noncommutative standard conjecture $C_{NC}(\cD_\perf^\dg(Z))$. Now, if by hypothesis the sign conjecture $C^+(Z)$ holds, there exist algebraic correspondences $\underline{\pi}^\pm_Z$ realizing the K\"unneth projectors $\pi^\pm_Z$. Hence, by taking for noncommutative correspondences $\pi^\pm_{\cD_\perf^\dg(Z)}$ the image of $\underline{\pi}^\pm_Z$ under the above functor \eqref{eq:compChow} we conclude that the noncommutative standard conjecture $C_{NC}(\cD_\perf^\dg(Z))$ also holds. This achieves the proof.
\section{Tannakian structure}\label{sec:Tannaka}
In this section we prove Theorem~\ref{thm:main2}. In what follows, $k$ will be a field of characteristic zero. 

\begin{proposition}\label{prop:aux2}
Assume that $F$ is a field extension of $k$ or vice-versa and that the noncommutative standard conjecture $C_{NC}(\cA)$ holds for a smooth and proper dg category $\cA$. Then, given any noncommutative Chow motive of shape $(\cA,e)$ (see \S\ref{sub:Chow}), the kernel of the induced surjective ring homomorphism
$$ \End_{\NChow(k)_F/\mathrm{Ker}(\overline{HP_\ast})}((\cA,e)) \twoheadrightarrow \End_{\NChow(k)_F/\cN}((\cA,e))$$
is a nilpotent ideal.
\end{proposition}
\begin{proof}
Let us consider first the case where $F$ is a field extension of $k$. By Theorem~\ref{thm:HP}, we have a $F$-linear symmetric monoidal functor
\begin{equation}\label{eq:aux1}
\overline{HP_\ast}: \NChow(k)_F \too \mathrm{sVect}(F)\,.
\end{equation}
The associated kernel $\mathrm{Ker}(\overline{HP_\ast})$ is a $\otimes$-ideal of $\NChow(k)_F$ and so we obtain an induced functor
\begin{equation}\label{eq:aux2}
\NChow(k)_F/\mathrm{Ker}(\overline{HP_\ast}) \too \mathrm{sVect}(F)
\end{equation}
which is not only $F$-linear and faithful but moreover symmetric monoidal. Recall from \S\ref{sub:Num} that the ideal $\cN$ is the largest $\otimes$-ideal of $\NChow(k)_F$ (distinct from the entire category). Hence, the induced functor
$$\NChow(k)_F/\mathrm{Ker}(\overline{HP_\ast}) \too \NChow(k)_F/\cN$$
is not only $F$-linear and symmetric monoidal but moreover full and (essentially) surjective. Now, observe that if by hypothesis the noncommutative standard conjecture $C_{NC}(\cA)$ holds for a smooth and proper dg category $\cA$, then the K\"unneth projectors
\begin{equation}\label{eq:Kunneth-proj}
\pi_{(\cA,e)}^{\pm}:  \overline{HP_\ast}((\cA,e)) \twoheadrightarrow \overline{HP_\ast}^{\pm}((\cA,e)) \hookrightarrow \overline{HP_\ast}((\cA,e))
\end{equation}
associated to a noncommutative Chow motive $(\cA,e)$ are also algebraic. Simply take for $\underline{\pi}^\pm_{(\cA,e)}$ the noncommutative correspondence $e\circ \underline{\pi}^\pm_\cA \circ e$. Let $\underline{X} \in \End_{\NChow(k)_F}((\cA,e))$ be a noncommutative correspondence. We need to show that if $\underline{X}$ becomes trivial in $\NChow(k)_F/\cN$, then it is nilpotent in $\NChow(k)_F/\mathrm{Ker}(\overline{HP_\ast})$. As explained above, the K{\"u}nneth projectors \eqref{eq:Kunneth-proj} can be written as $\pi^{\pm}_{(\cA,e)}=\overline{HP_\ast}(\underline{\pi}^{\pm}_{(\cA,e)})$, with $\underline{\pi}^{\pm}_{(\cA,e)} \in \End_{\NChow(k)_F}((\cA,e))$. If by hypothesis $\underline{X}$ becomes trivial in $\NChow(k)_F/\cN$, then by definition of $\cN$ the intersection numbers $\langle \underline{X} \cdot \underline{\pi}^{\pm}_{(\cA,e)}\rangle  \in K_0(k)_F \simeq F$ vanish. Moreover, since $\cN$ is a $\otimes$-ideal, the intersection numbers $\langle \underline{X}^n\cdot \underline{\pi}^{\pm}_{(\cA,e)}\rangle$ vanish also for any integer $n \geq 1$, where $\underline{X}^n$ stands for the $n^{\mathrm{th}}$-fold composition of $\underline{X}$. As proved in \cite[Corollary~4.4]{Semi}, the intersection numbers $\langle \underline{X}^n\cdot \underline{\pi}^{\pm}_{(\cA,e)}\rangle$ agree with the categorical trace of the noncommutative correspondences $\underline{X}^n \circ \underline{\pi}^\pm_{(\cA,e)}$. Since the above functor \eqref{eq:aux1} is symmetric monoidal we then conclude that the following traces
$$ \mathrm{tr}(\overline{HP_\ast}(\underline{X}^n \circ \underline{\pi}^\pm_{(\cA,e)})) =  \mathrm{tr}(\overline{HP_\ast}(\underline{X})^n \circ \pi^\pm_{(\cA,e)}) =  \mathrm{tr}(\overline{HP_\ast}^\pm(\underline{X})^n) \qquad n \geq 1$$
vanish. Recall that over a field of characteristic zero, a nilpotent linear map can be characterized by the fact that the trace of all its $n^{\mathrm{th}}$-fold compositions vanish. As a consequence, the $F$-linear transformations 
$$ \overline{HP_\ast}^\pm(\underline{X}):  \overline{HP_\ast}((\cA,e)) \twoheadrightarrow \overline{HP_\ast}^\pm((\cA,e)) \stackrel{\overline{HP_\ast}(\underline{X})}{\too} \overline{HP_\ast}^\pm((\cA,e)) \hookrightarrow \overline{HP_\ast}((\cA,e))$$
are nilpotent and so $\overline{HP_\ast}(\underline{X})$ is also nilpotent, Finally, since the functor \eqref{eq:aux2} is faithful, we conclude that the noncommutative correspondence $\underline{X}$ becomes nilpotent in $\NChow(k)_F/\mathrm{Ker}(\overline{HP_\ast})$. This achieves the proof.

The case where $k$ is a field extension of $F$ is similar. The only difference is that the functor \eqref{eq:aux2} take values in $\mathrm{sVect}(k)$ instead of $\mathrm{sVect}(F)$. Since by hypothesis $k$ is a field of characteristic zero, the same reasoning applies and so the proof is finished.
\end{proof}

\begin{proposition}\label{prop:auxiliar}
Assume that $F$ is a field extension of $k$ or vice-versa, and that the noncommutative standard conjecture $C_{NC}(\cA)$ holds for every smooth and proper dg category $\cA$. Then, the induced functor
\begin{equation}\label{eq:induced2}
(\NChow(k)_F/\mathrm{Ker}(\overline{HP_\ast}))^\natural \too \NNum(k)_F
\end{equation}
is full, conservative, and (essentially) surjective.
\end{proposition}
\begin{proof}
Since we have an inclusion $\mathrm{Ker}(\overline{HP_\ast})\subset \cN$ of $\otimes$-ideals, the induced functor 
$$ \NChow(k)_F/\mathrm{Ker}(\overline{HP_\ast}) \twoheadrightarrow \NChow(k)_F/\cN$$
is clearly full and (essentially) surjective. As explained in \cite[Prop.~3.2]{Ivorra}, idempotent elements can always be lifted along surjective $F$-linear homomorphisms with a nilpotent kernel. Hence, by Proposition~\ref{prop:aux2}, we conclude that the functor \eqref{eq:induced2} is also full and (essentially) surjective. The fact that it is moreover conservative follows from \cite[Lemma~3.1]{Ivorra}.
\end{proof}
In order to simplify the proof of Theorem~\ref{thm:main2}, we now introduce the following general result.
\begin{proposition}\label{prop:new}
Let $F$ be a field of characteristic zero, $K$ a field extension of $F$, and two $F$-linear symmetric monoidal functors
\begin{eqnarray*}
H: \cC \too \mathrm{sVect}(K) && P: \cC \too \cD\,,
\end{eqnarray*}
where $\mathrm{sVect}(K)$ denotes the category of finite dimensional super $K$-vector spaces. Assume that $H$ is faithful, that $P$ is (essentially) surjective, and that for every object $N \in \cC$, the K\"unneth projectors $\pi^\pm_N:H(N) \twoheadrightarrow H^\pm(N) \hookrightarrow H(N)$ can be written as $\pi^\pm_N=H(\underline{\pi}^\pm_N)$ with $\underline{\pi}^\pm_N \in \End_\cC(N)$. Then, by modifying the symmetry isomorphism constraints of $\cC$ and $\cD$, we obtain new symmetric monoidal categories $\cC^\dagger$ and $\cD^\dagger$ and (composed) $F$-linear symmetric monoidal functors
\begin{eqnarray*}
\cC^\dagger \stackrel{H}{\too} \mathrm{sVect}(K) \too \mathrm{Vect}(K) && P: \cC^\dagger \too \cD^\dagger\,,
\end{eqnarray*}
where $\mathrm{sVect}(K) \to \mathrm{Vect}(K)$ is the forgetful functor.
\end{proposition}
\begin{proof}
By applying \cite[Prop.~8.3.1]{AK} to the functor $H$ we obtain the symmetric monoidal category $\cC^\dagger$ and the $F$-linear symmetric monoidal (composed) functor $\cC^\dagger \stackrel{H}{\to}\mathrm{sVect}(K) \to \mathrm{Vect}(K)$. As explained in {\em loc.~cit.}, the new symmetry isomorphism constraints are given by 
\begin{equation}\label{eq:modif}
c^\dagger_{N_1,N_2}:= c_{N_1,N_2} \circ (e_{N_1} \otimes e_{N_2})\,,
\end{equation} 
where $e_N$ is the endomorphism $2\cdot \underline{\pi}^+_N-\id_N$ of $N$. Since by hypothesis the functor $P$ is (essentially) surjective, we can then use the image of the endomorphisms $e_N$ to modify the symmetry isomorphism constraints of $\cD$ as in the above formula \eqref{eq:modif}. We obtain in this way the symmetric monoidal category $\cD^\dagger$ and the $F$-linear symmetric monoidal functor $P: \cC^\dagger \to \cD^\dagger$.
\end{proof}
\subsection*{Proof of Theorem~\ref{thm:main2}}
Note first that since by hypothesis $k$ is of characteristic zero, the (non-trivial) field $F$ is also of characteristic zero. As explained in (the proofs of) Propositions~\ref{prop:aux2} and \ref{prop:auxiliar}, we have $F$-linear symmetric monoidal functors
\begin{eqnarray}
\overline{HP_\ast}: (\NChow(k)_F/\mathrm{Ker}(\overline{HP_\ast}))^\natural \too \mathrm{sVect}(K) \label{eq:Func1}\\
(\NChow(k)_F/\mathrm{Ker}(\overline{HP_\ast}))^\natural \too \NNum(k)_F \,, \label{eq:Func2}
\end{eqnarray}
with \eqref{eq:Func1} faithful and \eqref{eq:Func2} (essentially) surjective. The field $K$ is equal to $F$ when $F$ is a field extension of $k$ and equal to $k$ when $k$ is a field extension of $F$. In both cases, $K$ is a field extension of $F$. By hypothesis the noncommutative standard conjecture $C_{NC}(\cA)$ holds for every smooth and proper dg category $\cA$, and so we can apply the above general Proposition~\ref{prop:new} to the functors \eqref{eq:Func1} and \eqref{eq:Func2}. We obtain then the following diagram
\begin{equation}\label{eq:diagram3}
\xymatrix@C=3em@R=1.5em{
(\NChow(k)_F/\mathrm{Ker}(\overline{HP_\ast}))^{\natural,\dagger} \ar[r]^-{\overline{HP_\ast}} \ar[d] & \mathrm{sVect}(K) \ar[r] & \mathrm{Vect}(K)\\
\NNum^\dagger(k)_F & & \,.
}
\end{equation}
Now, recall that by construction the category $\NNum^\dagger(k)_F$ is $F$-linear, additive, rigid symmetric monoidal, and such that the endomorphisms of its $\otimes$-unit is the field $F$. By Theorem~\ref{thm:semi} it is moreover abelian (semi-simple). Hence, in order to prove that it is Tannakian, we can make use of Deligne's intrinsic characterization; see Theorem~\ref{TannTrpos}. Concretely, we need to show that the rank $\mathrm{rk}(N)$ of every noncommutative numerical motive $N \in \NNum^\dagger(k)_F$ is a non-negative integer. Since the vertical functor of diagram \eqref{eq:diagram3} is (essentially) surjective and restricts to an isomorphism between the endomorphisms of the corresponding $\otimes$-units, we conclude that $\mathrm{rk}(N)=\mathrm{rk}(\widetilde{N})$ for any lift $\widetilde{N} \in (\NChow(k)_F/\mathrm{Ker}(\overline{HP_\ast}))^{\natural,\dagger}$ of $N$. Moreover, since the (composed) horizontal functor of diagram \eqref{eq:diagram3} is faithful, we have $\mathrm{rk}(\widetilde{N})= \mathrm{rk}(\overline{HP_\ast}(\widetilde{N}))$. Finally, since in the symmetric monoidal category $\mathrm{Vect}(K)$ the rank $\mathrm{rk}(\overline{HP_\ast}(\widetilde{N}))$ can be written as $\mathrm{dim}(\overline{HP_\ast}^+(\widetilde{N})) + \mathrm{dim}(\overline{HP_\ast}^-(\widetilde{N}))$, we conclude that $\mathrm{rk}(N)$ is a non-negative integer. This achieves the proof. 

\section{Noncommutative homological motives}\label{sec:homological}
In this section we start by recalling the standard conjecture $D$. Then, we prove Theorems \ref{thm:equality1} and \ref{thm:neutral} and along the way introduced the category of {\em noncommutative homological motives}. In what follows, $Z$ will be a smooth projective $k$-scheme over a field of characteristic zero.

Recall from \cite[\S3.2]{Andre} the definition of the $F$-vector spaces $\cZ_{\mathrm{hom}}^\ast(Z)_F$ and $\cZ_{\mathrm{num}}^\ast(Z)_F$ of algebraic cycles (of arbitrary codimension), where the homological equivalence relation is taken with respect to de Rham cohomology. 

\smallbreak

{\bf Standard conjecture $D(Z)$ (see \cite[\S5.4.1.1]{Andre}):} {\it The following equality holds
$$ \cZ^\ast_{\mathrm{hom}}(Z)_F=\cZ^\ast_{\mathrm{num}}(Z)_F\,.$$
}
Now, recall from Theorem~\ref{thm:HP} that periodic cyclic homology gives rise to a well-defined $F$-linear symmetric monoidal functor
\begin{equation}\label{eq:func3}
\overline{HP_\ast}: \NChow(k)_F \too \mathrm{sVect}(K)\,.
\end{equation}
The field $K$ is equal to $F$ when $F$ is a field extension of $k$ and equal to $k$ when $k$ is a field extension of $F$. Given a smooth and proper dg category $\cA$, we then have an induced $F$-linear homomorphism
$$ K_0(\cA)_F = \Hom_{\NChow(k)_F}(\underline{k},\cA) \stackrel{\overline{HP_\ast}}{\too} \Hom_{\mathrm{sVect}(K)}(\overline{HP_\ast}(\underline{k}),\overline{HP_\ast}(\cA))\,.$$
The associated kernel gives rise to a well-defined equivalence relation on $K_0(\cA)_F$ that we denote by $\sim_{\mathrm{hom}}$. On the other hand, as explained in \S\ref{sub:Num}, the noncommutative correspondences which are numerically equivalent to zero form a $F$-linear subspace of $K_0(\cA)_F=\Hom_{\NChow(k)_F}(k,\cA)$ and hence give rise to an equivalence relation on $K_0(\cA)_F$ that we will denote by $\sim_{\mathrm{num}}$.
\subsection*{Proof of Theorem~\ref{thm:equality1}}
As explained in the proof of Theorem~\ref{thm:equality}, the composition
\begin{equation}\label{eq:composed}
\Chow(k)_F \stackrel{\tau}{\too} \Chow(k)_F/_{\!\!-\otimes \bbQ(1)} \stackrel{R}{\too} \NChow(k)_F \stackrel{\overline{HP_\ast}}{\too} \mathrm{sVect}(k)
\end{equation}
agrees with the functor $s\overline{H^\ast_{dR}}$ (see \eqref{eq:sHdR}) associated to de Rham cohomology. Hence, its kernel ${\it Ker}$ agrees with the kernel of the symmetric monoidal functor
$$ \overline{H^\ast_{dR}}: \Chow(k)_F \too \mathrm{GrVect}(k)\,.$$
From \eqref{eq:diagram-new} we then obtain the following commutative diagram
$$
\xymatrix@C=1.5em@R=1.5em{
\Chow(k)_F/{\it Ker} \ar[d] \ar[r]^-\Psi & (\Chow(k)_F/_{\!\!-\otimes \bbQ(1)})/\mathrm{Ker}(\overline{HP_\ast} \circ R) \ar[d] \ar[r]^-\Phi & \NChow(k)_F/\mathrm{Ker}(\overline{HP_\ast}) \ar[d] \\
\Num(k)_F \ar[r]_-{\tau} & \Num(k)_F/_{\!\!-\otimes \bbQ(1)} \ar[r]_-{R_\cN} & \NNum(k)_F\,,
}
$$
where {\it Ker} denotes the kernel of the composed functor \eqref{eq:composed}.
Now, by Lemma~\ref{lem:orbit-kernel} (with $\cC=\Chow(k)_F$, $\cO=\bbQ(1)$, and $H=\overline{HP_\ast}\circ R$), we observe that the functor $\Psi$ admits the following factorization
$$ \Chow(k)_F/{\it Ker} \stackrel{\tau}{\too} (\Chow(k)_F/{\it Ker})/_{\!\!-\otimes \bbQ(1)} \stackrel{\Omega}{\too} (\Chow(k)_F/_{\!\!-\otimes \bbQ(1)})/\mathrm{Ker}(\overline{HP_\ast} \circ R)\,.$$
As explained in the proof of Theorem~\ref{thm:equality}, the image of $Z$ under the composed functor $\Phi \circ \Psi$ is naturally isomorphic to the dg category $\cD_\perf^\dg(Z)$. Similarly, the image of the affine $k$-scheme $\mathrm{spec}(k)$ is naturally isomorphic to $\cD_\perf^\dg(\mathrm{spec}(k))\simeq k$. We can then consider the following induced commutative square
$$
\xymatrix{
(\Chow(k)_F/{\it Ker})/_{\!\!-\otimes \bbQ(1)} \ar[r] \ar[d] \ar[r]^-{\Phi \circ \Omega} & \NChow(k)_F/\mathrm{Ker}(\overline{HP_\ast}) \ar[d] \\
\Num(k)_F/_{\!\!-\otimes \bbQ(1)} \ar[r]_-{R_{\cN}} & \NNum(k)_F
}
$$
and the associated commutative diagram
$$
\xymatrix{
\Hom_{(\Chow(k)_F/\mathrm{Ker})/_{\!\!-\otimes \bbQ(1)}}(\mathrm{spec}(k),Z) \ar[d] \ar[r] & \Hom_{\NChow(k)_F/\mathrm{Ker}(\overline{HP_\ast})}(k,\cD_\perf^\dg(Z)) \ar[d] \\
\Hom_{\Num(k)_F/_{\!\!-\otimes \bbQ(1)}}(\mathrm{spec}(k),Z) \ar[r] & \Hom_{\NNum(k)_F}(k,\cD_\perf^\dg(Z))\,.
}
$$
By combining the above arguments with the constructions of the categories $\Chow(k)_F$, $\Num(k)_F$, $\NChow(k)_F$ and $\NNum(k)_F$, we then observe that the preceding commutative square identifies with 
\begin{equation}\label{eq:diag-key}
\xymatrix{
\cZ^\ast_{\mathrm{hom}}(Z)_F \ar@{->>}[d] \ar@{->>}[r] & K_0(\cD_\perf^\dg(Z))_F/\!\!\sim_{\mathrm{hom}} \ar[d] \\
\cZ^\ast_{\mathrm{num}}(Z)_F \ar[r]_-{\sim} & K_0(\cD_\perf^\dg(Z))_F/\!\!\sim_{\mathrm{num}}\,.
}
\end{equation}
Since the functor $R_{\cN}$ is fully-faithful, the lower horizontal map is an isomorphism. Recall from Lemma~\ref{lem:orbit-kernel}, that the functor $\Omega$ is full. Hence, since the induced functor $\Phi$ is fully-faithful, the upper horizontal map in \eqref{eq:diag-key} is surjective. Now, suppose that the standard conjecture $D(Z)$ holds, \ie that the left vertical map of diagram \eqref{eq:diag-key} is an isomorphism. The commutativity of the diagram combined with all the above facts implies that the right vertical map in \eqref{eq:diag-key} is injective. Since by construction it is already surjective, we conclude that it is an isomorphism. This is precisely the statement of the noncommutative standard conjecture $D_{NC}(\cD_\perf^\dg(Z))$ and so the proof is finished.

\medbreak

In analogy with the commutative world we introduce the category of noncommutative homological motives.

\begin{definition}\label{def:homological}
Let $k$ and $F$ be fields with $F$ a field extension of $k$ or vice-versa. The category $\NHom(k)_F$ of {\em noncommutative homological motives} is the pseudo-abelian envelope of the quotient category $\NChow(k)_F/\mathrm{Ker}(\overline{HP_\ast})$, where $\mathrm{Ker}(\overline{HP_\ast})$ is the kernel of the functor \eqref{eq:func3}.
\end{definition}
As explained in \S\ref{sub:Num}, the ideal $\cN$ is the largest $\otimes$-ideal of $\NChow(k)_F$ (distinct from the entire category). Hence, the inclusion of $\otimes$-ideals $\mathrm{Ker}(\overline{HP_\ast}) \subset \cN$ gives rise to an induced functor
\begin{equation}\label{eq:HN}
\NHom(k)_F \too \NNum(k)_F
\end{equation}
\begin{proposition}\label{prop:new-new}
Let $k$ be a field of characteristic zero. Then, if the noncommutative standard conjectures $C_{NC}(\cA)$ and $D_{NC}(\cA)$ hold for every smooth and proper dg category $\cA$, the induced functor \eqref{eq:HN} is an equivalence of categories.
\end{proposition}
\begin{proof}
By Proposition~\ref{prop:auxiliar} the above functor \eqref{eq:HN} is full and (essentially) surjective. It remains then to show that it is faithful. Given noncommutative Chow motives $(\cA,e)$ and $(\cB,e')$ (see \S\ref{sub:Chow}) we need to show that the induced $F$-linear homomorphism
$$\Hom_{\NChow(k)_F/\mathrm{Ker}(\overline{HP_\ast})}((\cA,e),(\cB,e')) \too \Hom_{\NNum(k)_F}((\cA,e),(\cB,e'))$$
is faithful. For this, it suffices to show that the induced $F$-linear homomorphism
\begin{equation}\label{eq:induced7}
\Hom_{\NChow(k)_F/\mathrm{Ker}(\overline{HP_\ast})}(\cA,\cB) \too \Hom_{\NNum(k)_F}(\cA,\cB)
\end{equation}
is faithful. The categorical dual of $\cA$ is its opposite dg category $\cA^\op$, and so by adjunction, \eqref{eq:induced7} identifies with the homomorphism
$$\Hom_{\NChow(k)_F/\mathrm{Ker}(\overline{HP_\ast})}(\underline{k}, \cA^\op \otimes_k \cB) \too \Hom_{\NNum(k)_F}(\underline{k}, \cA^\op \otimes_k \cB)\,.
$$
By the definition of the categories $\NChow(k)_F$ and $\NNum(k)_F$, we observe that the above homomorphism identifies with the canonical homomorphism
\begin{equation}\label{eq:iso-final}
K_0(\cA^\op \otimes_k \cB)_F/\!\!\sim_{\mathrm{hom}} \too K_0(\cA^\op \otimes_k \cB)_F/\!\!\sim_{\mathrm{num}}\,.
\end{equation}
Finally, since by hypothesis the noncommutative standard conjecture $D_{NC}(\cA^\op\otimes \cB)$ holds, \ie \eqref{eq:iso-final} is an isomorphism, we conclude that the homomorphism \eqref{eq:induced7} is faithful. This achieves the proof.
\end{proof}
\subsection*{Proof of Theorem~\ref{thm:neutral}}
Recall from the proof of Theorem~\ref{thm:main2} the construction of the following diagram
$$
\xymatrix@C=3em@R=1.5em{
\NHom^\dagger(k)_F \ar[r]^-{\overline{HP_\ast}} \ar[d] & \mathrm{sVect}(F)  \ar[r] & \mathrm{Vect}(F) \\
\NNum^\dagger(k)_F & &\,.
}
$$
Proposition~\ref{prop:new-new} implies that the vertical functor in the above diagram is an equivalence of categories. As a consequence, we obtain an exact faithful $\otimes$-functor
\begin{equation}\label{eq:functor}
\overline{HP_\ast}: \NNum^\dagger(k)_F \too \mathrm{Vect}(F)\,.
\end{equation} 
Recall from Theorem~\ref{thm:main2} that the category $\NNum^\dagger(k)_F$ is Tannakian. The above functor \eqref{eq:functor} allows us to conclude that $\NNum^\dagger(k)_F$ is moreover neutral and so the proof is finished.
\section{Motivic Galois groups}\label{sec:Galois}
In this section we prove Theorem~\ref{thm:main3}. At this point the reader is invited to consult Appendix~\ref{appendix} as we will make full use of all its notions and results.
\begin{proposition}\label{prop:technical1}
Let $F$ be a field of characteristic zero and $\cT=(\cC,w,T)$ a neutral Tate triple.
Let us denote by $\cS$ the full neutral Tannakian subcategory of $\cC$ generated by the Tate object $T$. Then, the pseudo-abelian envelope of the orbit category $\cC/_{\!\!-\otimes T}$ (see Appendix~\ref{app:orbit}) is a neutral Tannakian category and the sequence of exact $\otimes$-functors $\cS \subset \cC \to (\cC/_{\!\!-\otimes T})^\natural$ induces a group scheme isomorphism
$$ \mathrm{Gal}((\cC/_{\!\!-\otimes T})^\natural)  \stackrel{\sim}{\too} \mathrm{Ker}\left(t:  \mathrm{Gal}(\cC) \twoheadrightarrow \bbG_m \right) \,. $$
\end{proposition}
\begin{proof}
By Theorem~\ref{thm:quot1Tann}(ii) the inclusion $\cS\subset \cC$ gives rise to a surjective group homomorphism $t: \mathrm{Gal}(\cC) \twoheadrightarrow \mathrm{Gal}(\cS)$. Let us denote by $H$ its kernel. Thanks to items (i) and (iii) of Theorem~\ref{thm:quot1Tann}, we have a sequence of exact $\otimes$-functors $\cS\subset \cC \to \mathrm{Rep}_F(H)$ inducing a group scheme isomorphism
$$ \mathrm{Gal}(\mathrm{Rep}_F(H))  \stackrel{\sim}{\too} \mathrm{Ker}\left(t:  \mathrm{Gal}(\cC) \twoheadrightarrow \mathrm{Gal}(\cS) \right) \,. $$
The proof will consist then on showing that the categories $\mathrm{Rep}_F(H)$ and $(\cC/_{\!\!-\otimes T})^\natural$ are $\otimes$-equivalent and that the Galois group $\mathrm{Gal}(\cS)$ agrees with the multiplicative group $\bbG_m$. Let us start with the latter claim. The $\bbZ$-grading $w$ of the Tate triple structure implies that the fiber functor $\omega: \cC \to \mathrm{Vect}(F)$ factors as follows  
$$
\xymatrix{
\cC \ar[r]^-{\overline{\omega}} \ar[dr]_-\omega & \mathrm{GrVect}(F) \ar[d]^-U\\
& \mathrm{Vect}(F)\,,
}
$$
where $\mathrm{GrVect}(F)$ denotes the category of finite dimensional $\bbZ$-graded $F$-vector spaces and $U$ the forgetful functor $\{V_n\}_{n \in \bbZ} \mapsto \oplus_{n \in \bbZ} V_n$. Since the Tate object $T$ is invertible and the functor $\overline{\omega}$ is symmetric monoidal, the object $\overline{\omega}(T)$ is also invertible. Moreover, since $T$ is of degree two and the invertible objects in $\mathrm{GrVect}(F)$ are one-dimensional, the $\bbZ$-graded $F$-vector space $\overline{\omega}(T)$ is also of degree two and one-dimensional. In particular, $\End_{\mathrm{GrVect}(F)}(\overline{\omega}(T))\simeq F$. This implies that the functor $\overline{\omega}$ establishes a $\otimes$-equivalence between $\cS$ and the Tannakian category generated by $\overline{\omega}(T)$. The latter category identifies with the subcategory $\mathrm{GrVect}^+(F)\subset \mathrm{GrVect}(F)$ of evenly supported $\bbZ$-graded $F$-vector spaces. Via a re-numerotation of the indexes, this category is $\otimes$-equivalent to $\mathrm{GrVect}(F)$. As a consequence, the Galois group $\mathrm{Gal}(\cS)$ (with respect to $\omega$) agrees with the Galois group of the Tannakian category $\mathrm{GrVect}(F)$ (with respect to $U$). It is well known that the latter Galois group agrees with the multiplicative group $\bbG_m$ and so the proof of the above claim is finished.

Let us now show that the categories $\mathrm{Rep}_F(H)$ and $(\cC/_{\!\!-\otimes T})^\natural$ are $\otimes$-equivalent. Since by hypothesis $\cT=(\cC,w,T)$ is a Tate triple, Lemma~\ref{G0lem}(ii) (with $G=\mathrm{Gal}(\cC)$ and $G_0=H$) implies that the quotient functor $Q: \cC \to \mathrm{Rep}_F(H)$ maps the Tate object $T$ to the $\otimes$-unit of $\mathrm{Rep}_F(H)$. As a consequence, we have a natural $2$-isomorphism $Q\circ (-\otimes T) \stackrel{\sim}{\Rightarrow} Q$. By the $2$-universality property of the orbit category $\cC/_{\!\!-\otimes T}$ and the fact that the category $\mathrm{Rep}_F(H)$ is idempotent complete, we obtain then an induced symmetric monoidal functor
\begin{equation}\label{eq:induced-new}
(\cC/_{\!\!-\otimes T})^\natural \too \mathrm{Rep}_F(H)\,.
\end{equation}
Now, recall from \cite[Example~2.6]{Milne} that since $(\cC,w,T)$ is a Tate triple, the category $\mathrm{Rep}_F(H)$ can be identified with the quotient category $\cC/\omega_0$ with respect to the $F$-valued fiber functor
\begin{eqnarray*}
\omega_0: \cS \too \mathrm{Vect}(F) && X \mapsto \mathrm{colim}_n \, \Hom_\cC\left(\bigoplus_{r=-n}^{n} {\bf 1}(r),X\right)\,.
\end{eqnarray*}
As explained in Definition~\ref{quotTann}, $\cC/\omega_0$ is the pseudo abelian envelope of a certain category $\cC'$. Hence, it suffices to show that $\cC/_{\!\!-\otimes T}$ and $\cC'$ are $\otimes$-equivalent. By construction, they have the same objects. In what concerns their morphisms, by arguing as in \cite[Prop.~2.3]{Milne}, we conclude that 
\begin{eqnarray*}
\Hom_{\cC'}(X,Y) & =& \mathrm{colim}_n \bigoplus_{r=-n}^n \Hom_\cC( {\bf 1},\underline{\mathrm{Hom}}(X\otimes T^{\otimes r},Y)) \\
& = &  \mathrm{colim}_n \bigoplus_{r=-n}^n \Hom_\cC(X\otimes T^{\otimes r} ,Y)\\
& =& \bigoplus_{j\in \bbZ} \Hom_\cC(X,Y\otimes T^{\otimes j})\,.
\end{eqnarray*}
The latter description agrees with the one of the orbit category $\cC/_{\!\!-\otimes T}$; see \eqref{eq:comp} with $\cO=T$. As a consequence, the above induced functor \eqref{eq:induced-new} is a $\otimes$-equivalence and so the proof is finished. 
\end{proof}
The notion of Tate triple does not admit an immediate generalization to the super-Tannakian setting. Motivated by the standard theory of motives, we nevertheless introduce the following notion.
\begin{definition}\label{sTate3def}
A {\em super-Tate triple} $\cS\cT=(\cC, \omega, \underline{\pi}^\pm_X, \cT^\dagger)$ consists of:
\begin{itemize}
\item[(i)] a neutral super-Tannakian category $\cC$;
\item[(ii)] a super-fiber functor $\omega: \cC \to \mathrm{sVect}(F)$;
\item[(iii)] idempotent endomorphisms $\underline{\pi}^\pm_X \in \End_\cC(X), X \in \cC$, such that $\omega(\underline{\pi}^\pm_X)=\pi_X^\pm$, where $\pi_X^\pm$ are the K\"unneth projectors;
\item[(iv)] a neutral Tate triple $\cT^\dagger=(\cC^\dagger, w,T)$ on the category $\cC^\dagger$. Recall from Proposition~\ref{prop:new} that $\cC^\dagger$ is obtained from $\cC$ by modifying the symmetry isomorphism constraints throughout the use of the endomorphisms $\underline{\pi}_X^\pm$. 
\end{itemize} 
\end{definition}
\begin{example}\label{ex:new}
Let $k$ be a field of characteristic zero. If we assume that the standard conjectures $C(Z)$ and $D(Z)$ hold for every smooth projective $k$-scheme $Z$, then the data
$$ (\Num(k)_k, s\overline{H^\ast_{dR}}, \underline{\pi}_X^\pm, (\Num^\dagger(k)_k, w, \bbQ(1)))$$
is a super-Tate triple. The super-fiber functor $s\overline{H^\ast_{dR}}: \Num(k)_k \to \mathrm{sVect}(k)$ is the one associated to de Rham cohomology and $(\Num^\dagger(k)_k,w, \bbQ(1))$ is the neutral Tate triple obtained from the one of Example~\ref{ex:Tate}(ii) by extension of scalars along the functor $-\otimes_\bbQ k$; see \cite[\S1]{Milne}.
\end{example}
\begin{proposition}\label{sTate3prop}
Let $F$ be a field of characteristic zero and $\cS\cT=(\cC, \omega, \underline{\pi}^\pm_X,\cT^\dagger)$ a super-Tate triple. Let $\cS$ denote the full neutral super-Tannakian subcategory of $\cC$ generated by the Tate object $T$. Assume the following conditions:
\begin{itemize}
\item[(i)] The Tate object $T$ is such that $\underline{\pi}^-_T(T)=0$.
\item[(ii)] Let $\epsilon: \mu_2 \to H$ be the $\bbZ/2$-grading induced by the $\bbZ$-grading on the neutral Tate triple $\cT^\dagger$ as in Lemma~\ref{G0lem} (with $G=\mathrm{Gal}(\cC^\dagger)$ and $G_0=H$). Then, the affine super-group scheme $(H,\epsilon)$ is isomorphic to the kernel of the induced surjective group homomorphism $\mathrm{sGal}(\cC) \twoheadrightarrow \mathrm{sGal}(\cS)$.
\item[(iii)] The super-Tannakian category $\mathrm{Rep}_F((H,\epsilon))$ of finite dimensional $F$-valued super representations is such that $\mathrm{Rep}^\dagger_F((H,\epsilon))\simeq Q$, where $Q \simeq \mathrm{Rep}_F(H)$ is the quotient Tannakian category associated to the inclusion $\cS^\dagger \subset \cC^\dagger$; see items (ii) and (iii) of Theorem~\ref{thm:quot1Tann}. 
\end{itemize}
Then, the pseudo-abelian envelope of the orbit category $\cC/_{\!\!-\otimes T}$ is a neutral super-Tannakian category and the sequence of exact $\otimes$-functors $\cS \subset \cC \to (\cC/_{\!\!-\otimes T})^\natural$ induces a group scheme isomorphism
$$ \mathrm{sGal}((\cC/_{\!\!-\otimes T})^\natural)  \stackrel{\sim}{\too} \mathrm{Ker}\left(t:  \mathrm{sGal}(\cC) \twoheadrightarrow \bbG_m \right) \,. $$
\end{proposition}
\begin{proof}
By condition (ii), the sequence of exact of $\otimes$-functors $\cS \subset \cC \to \mathrm{Rep}_F((H,\epsilon))$ induces a group scheme isomorphism
$$\mathrm{sGal}(\mathrm{Rep}_F((H,\epsilon)) \stackrel{\sim}{\too} \mathrm{Ker}(t: \mathrm{sGal}(\cC) \twoheadrightarrow \mathrm{sGal}(\cS))\,.$$
Hence, the proof will consist on showing that the categories $\mathrm{Rep}_F((H,\epsilon))$ and $\cC/_{\!\!-\otimes T}$ are $\otimes$-equivalent and that the super-Galois group $\mathrm{sGal}(\cS)$ agrees with $\bbG_m$. Let us start with the latter claim. Recall from the proof of Proposition~\ref{prop:new} that the category $\cS^\dagger$ is obtained from $\cS$ by modifying its symmetry isomorphism constraints $c_{N_1,N_2}$. The new symmetry isomorphism constraints are given by
$$ c_{N_1,N_2}^\dagger:= c_{N_1,N_2} \circ (e_{N_1} \otimes e_{N_2})$$
where $e_N=2 \cdot \underline{\pi}^+_N-\id_N$. Since by hypothesis $\underline{\pi}^-_T(T)=0$ we conclude that $\underline{\pi}_N^-(N)=0$ for every object $N\in\cS$. As a consequence, $\underline{\pi}^+_N=\id_N$ and so $e_N=\id_N$. The symmetric monoidal category $\cS^\dagger$ is then equal to $\cS$. This implies that $\mathrm{sGal}(\cS)\simeq \mathrm{Gal}(\cS^\dagger)$ since the super-fiber functor $\omega: \cS \to \mathrm{sVect}(F)$ take values in the subcategory $\mathrm{sVect}^+(F)\subset \mathrm{sVect}(F)$ of evenly supported super $F$-vector spaces, which via the forgetful functor is $\otimes$-equivalent to $\mathrm{Vect}(F)$. As explained in the proof of Proposition~\ref{prop:technical1}, $\mathrm{Gal}(\cS^\dagger)\simeq \bbG_m$ and so we conclude that $\mathrm{sGal}(\cS) \simeq\bbG_m$.   

Let us now show that the categories $\mathrm{Rep}_F((H,\epsilon))$ and $\cC/_{\!\!-\otimes T}$ are $\otimes$-equivalent. As in the Tannakian case, we have a sequence of exact $\otimes$-functors $\cS \subset \cC \to \mathrm{Rep}_F((H,\epsilon))$ such that the Tate object $T$ is mapped to the $\otimes$-unit of $\mathrm{Rep}_F((H,\epsilon))$. By the $2$-universality property of the orbit category category $\cC/_{\!\!-\otimes T}$ and the fact that the category $\mathrm{Rep}_F((H,\epsilon))$ is idempotent complete, we obtain an induced symmetric monoidal functor
\begin{equation}\label{eq:induced-new2}
(\cC/_{\!\!-\otimes T})^\natural \too \mathrm{Rep}_F((H,\epsilon))\,.
\end{equation}
By combining Lemma~\ref{lem:new-technical} (with $\cO=T$) with condition (iii), we then obtain an induced functor
\begin{equation}\label{eq:2}
(\cC^\dagger/_{\!\!-\otimes T})^\natural \stackrel{\sim}{\too} (\cC/_{\!\!-\otimes T})^{\natural, \dagger} \too \mathrm{Rep}_F^\dagger((H,\epsilon)) \stackrel{\sim}{\too} \mathrm{Rep}_F(H)\,.
\end{equation}
Now, since by hypothesis $\cT^\dagger$ is a neutral Tate triple, the proof of Proposition~\ref{prop:technical1} shows us that \eqref{eq:2} is a $\otimes$-equivalence. We then conclude that \eqref{eq:induced-new2} is also a $\otimes$-equivalence and so the proof is finished.
\end{proof}
\subsection*{Proof of Theorem~\ref{thm:main3}}
Let us start by constructing the surjective group homomorphism \eqref{eq:surj1}. Since by hypothesis $k$ is a field of characteristic zero, Theorem~\ref{thm:main1} implies that $\NNum(k)_k$ is super-Tannakian. Moreover, since the noncommutative standard conjectures $C_{NC}(\cA)$ and $D_{NC}(\cA)$ hold for every smooth and proper dg category $\cA$, Proposition~\ref{prop:new-new} implies that periodic cyclic homology gives rise to a super-fiber functor
\begin{equation}\label{eq:3}
\overline{HP_\ast}: \NNum(k)_k \too \mathrm{sVect}(k)\,.
\end{equation}
Now, recall from \eqref{eq:diagram-new} the following composition
\begin{equation}\label{eq:4}
\Num(k)_k \stackrel{\tau}{\too} \Num(k)_k/_{\!\!-\otimes \bbQ(1)} \stackrel{R_\cN}{\too} \NNum(k)_k \,.
\end{equation}
By composing \eqref{eq:3} with \eqref{eq:4} we obtain then a well-defined super-fiber functor on $\Num(k)_k$ which, as explained in the proof of Theorem~\ref{thm:equality}, is given by
\begin{eqnarray*}
s\overline{H^\ast_{dR}}: \Num(k)_k \too \mathrm{sVect}(k) && Z \mapsto \left(\bigoplus_{n \even} H^n_{dR}(Z), \bigoplus_{n \odd} H^n_{dR}(Z) \right)\,.
\end{eqnarray*} 
Now, recall from Example~\ref{ex:new} that since by hypothesis the standard conjectures $C(Z)$ and $D(Z)$ hold for every smooth projective $k$-scheme $Z$, we have a super-Tate triple
$$ (\Num(k)_k, s\overline{H^\ast_{dR}}, \underline{\pi}_X^\pm, (\Num^\dagger(k)_k, w, \bbQ(1)))\,.$$
Let us now show that this super-Tate triple satisfies conditions (i)-(iii) of Proposition~\ref{sTate3prop}. In what concerns item (i), the $\bbZ$-graded $k$-vector space $\overline{H^\ast_{dR}}(\bbQ(1))$ is of degree two and one dimensional. This implies that $\underline{\pi}^+_{\bbQ(1)}=\id_{\bbQ(1)}$ and so $\underline{\pi}^-_{\bbQ(1)}(\bbQ(1))=0$. Item (ii) follows from \cite[Example~0.4]{Deligne} and \cite[Example~5.4]{DelMil} since the Tate triple $(\Num^\dagger(k)_k,w,\bbQ(1))$ can be written as $(\mathrm{Rep}(G), w, \bbQ(1))$ with $G=\mathrm{Gal}(\Num^\dagger(k)_k)$. In what concerns item (iii), recall from Example~\ref{ex:Tate}(ii) that the $\bbZ$-grading $w$ on $\Num^\dagger(k)_k$ is induced by the algebraic correspondences $\underline{\pi}_Z^n$ such that $\overline{H^\ast_{dR}}(\underline{\pi}^n_Z)=\pi_Z^n$. This implies that the $\bbZ/2$-grading $\epsilon$ agrees with the $\bbZ/2$-grading induced by the algebraic correspondences $\underline{\pi}_Z^\pm$ such that $s\overline{H^\ast_{dR}}(\underline{\pi}^\pm_Z)=\pi_Z^\pm$. As a consequence, we conclude that the category $\mathrm{Rep}((H,\epsilon))$ is obtained from $\mathrm{Rep}(H)$ by modifying its symmetry isomorphism constraints with respect to the $\bbZ/2$-grading induced by the algebraic correspondences $\underline{\pi}_Z^\pm$. Now, by Proposition~\ref{sTate3prop} we obtain an induced group isomorphism
\begin{equation}\label{eq:star1}
\mathrm{sGal}((\Num(k)_k/_{\!\!-\otimes \bbQ(1)})^\natural) \stackrel{\sim}{\too} \mathrm{Ker}(t: \mathrm{sGal}(\Num(k)_k) \twoheadrightarrow \bbG_m)\,.
\end{equation}
On the other hand since by construction the category $\NNum(k)_k$ is idempotent complete, the functor $R_{\cN}$ in diagram \eqref{eq:4} gives rise to a surjective group homomorphism
$$ \mathrm{sGal}(\NNum(k)_k) \twoheadrightarrow \mathrm{sGal}((\Num(k)_k /_{\!\!-\otimes \bbQ(1)})^\natural)\,.$$
By combining it with \eqref{eq:star1} we obtain finally the surjective group homomorphism \eqref{eq:surj1} of Theorem~\ref{thm:main3}.

Let us now construct the surjective group homomorphism \eqref{eq:surj2}. We start by showing that throughout the modification of the symmetry isomorphism constraints of the category $\NNum(k)_k$ (as in Theorem~\ref{thm:main2}), the sequence of functors
$$ \Num(k)_k \stackrel{\tau}{\too} \Num(k)_k/_{\!\!-\otimes \bbQ(1)} \stackrel{R_\cN}{\too} \NNum(k)_k$$
described in diagram~\eqref{eq:diagram-new} gives rise to the following sequence
\begin{equation}\label{eq:sequence}
\Num^\dagger(k)_k \stackrel{\tau}{\too} \Num^\dagger(k)_k/_{\!\!-\otimes \bbQ(1)} \stackrel{R_\cN}{\too} \NNum^\dagger(k)_k\,.
\end{equation}
In order to show this, consider the following diagram
$$
\xymatrix@C=2em@R=2em{
\Chow(k)_k \ar[r] \ar[d] & \Chow(k)_k/_{\!\!-\otimes \bbQ(1)} \ar[d] \ar[r] & \NChow(k)_k \ar[d] \ar[r]^-{\overline{HP_\ast}} & \mathrm{sVect}(k) \ar@{=}[d]\\
\Chow(k)_\bbQ/\mathrm{Ker} \ar[d] \ar[r] & (\Chow(k)_k/_{\!\!-\otimes \bbQ(1)})/\mathrm{Ker} \ar[r] \ar[d] & \NHom(k)_F \ar[r]_{\overline{HP_\ast}} \ar[d] & \mathrm{sVect}(k)\\
\Num(k)_k \ar[r] & \Num(k)_k/_{\!\!-\otimes \bbQ(1)} \ar[r] & \NNum(k)_k& \,,
}
$$
where $\mathrm{Ker}$ stands for the kernel of the respective composed horizontal functor. As explained in the proof of Theorem~\ref{thm:equality}, the upper horizontal composition corresponds to the functor $s\overline{H^\ast_{dR}}$; see \eqref{eq:sHdR}.
Since by hypothesis the standard conjecture $C(Z)$ (and hence the sign conjecture $C^+(Z)$) holds for all smooth projective $k$-schemes $Z$, the functors
\begin{eqnarray*}
\Chow(k)_k/\mathrm{Ker} \too \mathrm{sVect}(k) && \Chow(k)_k/\mathrm{Ker} \too \Num(k)_k
\end{eqnarray*} 
satisfy the conditions of the general Proposition~\ref{prop:new}. Proposition~\ref{prop:new} is clearly functorial on $\cC$ and so when applied to the two lower rows of the above commutative diagram, gives rise to the following sequence
\begin{equation}\label{eq:seq8}
\Num^\dagger(k)_k \too (\Num(k)_k/_{\!\!-\otimes \bbQ(1)})^\dagger \too \NNum^\dagger(k)_k\,.
\end{equation}
By Lemma~\ref{lem:new-technical} (with $\cC=\Num(k)_k$ and $\cO=\bbQ(1)$), we have a canonical $\otimes$-equivalence
$$ \Num^\dagger(k)_k/_{\!\!-\otimes \bbQ(1)} \stackrel{\sim}{\too} (\Num(k)_k/_{\!\!-\otimes \bbQ(1)})^\dagger $$
and so the above sequence \eqref{eq:seq8} reduces to \eqref{eq:sequence}. Now, recall from Example~\ref{ex:Tate}(ii) that $(\Num^\dagger(k)_k,w,\bbQ(1))$ is a neutral Tate triple whose fiber functor is induced from the functor $\overline{H^\ast_{dR}}$ associated to de Rham cohomology. By Proposition~\ref{prop:technical1} we obtain then an induced group isomorphism
\begin{equation}\label{eq:star}
\mathrm{sGal}((\Num^\dagger(k)_k/_{\!\!-\otimes \bbQ(1)})^\natural) \stackrel{\sim}{\too} \mathrm{Ker}(t: \mathrm{sGal}(\Num^\dagger(k)_k) \twoheadrightarrow \bbG_m)\,.
\end{equation}
On the other hand since the category $\NNum(k)_k$ is idempotent complete, the functor $R_{\cN}$ in diagram \eqref{eq:sequence} gives rise to a surjective group homomorphism
$$ \mathrm{sGal}(\NNum^\dagger(k)_k) \twoheadrightarrow \mathrm{sGal}((\Num^\dagger(k)_k /_{\!\!-\otimes \bbQ(1)})^\natural)\,.$$
By combining it with \eqref{eq:star} we obtain finally the surjective group homomorphism \eqref{eq:surj2} of Theorem~\ref{thm:main3}. This achieves the proof.

\appendix

\section{Tannakian formalism}\label{appendix}
In this appendix we collect the notions and results of the (super-)Tannakian formalism which are used throughout the article. In what follows, $F$ will be a field and $(\cC,\otimes, {\bf 1})$ a $F$-linear, abelian, rigid symmetric monoidal category, such that $\End_\cC({\bf 1})\simeq F$.

Recall from Deligne~\cite{Deligne-Fest} that a {\em $K$-valued fiber functor on $\cC$} (with $K$ a field extension of $F$) is an exact faithful $\otimes$-functor $\omega: \cC \to \mathrm{Vect}(K)$ with values in the category of finite dimensional $K$-vector spaces. The category $\cC$ is called {\em Tannakian} if it admits a $K$-valued fiber functor. It is called {\em neutral Tannakian} if it admits a $F$-valued fiber functor. In this latter case, the fiber functor $\omega: \cC \to \mathrm{Vect}(F)$ gives rise to a $\otimes$-equivalence between $\cC$ and the category $\mathrm{Rep}_F(\mathrm{Gal}(\cC))$ of finite dimensional $F$-valued representations of the {\em Galois group} $\mathrm{Gal}(\cC):=\underline{\mathrm{Aut}}^\otimes(\omega)$.

\begin{theorem}\label{TannTrpos} {\rm (Deligne's intrinsic characterization \cite{Deligne-Fest})}
Let $F$ be a field of characteristic zero. Then, $\cC$ is Tannakian if and only if the rank $\mathrm{rk}(N):=\mathrm{tr}(\id_N)$ of each one of its objects $N$ is a non-negative integer.
\end{theorem}

Following Deligne~\cite{Deligne}, a {\em $K$-valued super-fiber functor on $\cC$} (with $K$ a field extension of $F$) is an exact faithful $\otimes$-functor $\omega: \cC \to \mathrm{sVect}(K)$ with values in the category of finite dimensional super $K$-vector spaces. The category $\cC$ is called {\em super-Tannakian} if it admits a $K$-valued super-fiber functor. It is called {\em neutral super-Tannakian} if it admits a $F$-valued super-fiber functor. In this latter case, the super-fiber functor $\omega: \cC \to \mathrm{sVect}(F)$ gives rise to a $\otimes$-equivalence between $\cC$ and the category $\mathrm{Rep}_F(\mathrm{sGal}(\cC),\epsilon)$ of finite dimensional $F$-valued super-representations of the {\em super-Galois group} $\mathrm{sGal}(\cC):=\underline{\mathrm{Aut}}^\otimes(\omega)$, where $\epsilon$ is the parity automorphism implementing the super symmetry of $\omega$.  

\begin{theorem}\label{sutannDel} {\rm (Deligne's intrinsic characterization \cite{Deligne})}
Let $F$ be a field of characteristic zero. Then, $\cC$ is super-Tannakian if and only if is Schur-finite (see \S\ref{sec:Schur}). If $F$ is moreover algebraically closed, then $\cC$ is neutral super-Tannakian if and only if is Schur-finite.
\end{theorem}

\subsection{Tate triples}\label{gradingsec} {(consult Deligne-Milne~\cite[\S5]{DelMil})}
In this subsection, we assume that $F$ is of characteristic zero. Let $\mu_n$ be the affine group scheme of the {\em $n^{\mathrm{th}}$-roots of unity}, \ie the
kernel of the $n^{\mathrm{th}}$-power homomorphism $\bbG_m \to \bbG_m$. In particular, $\mu_2$ is the affine group scheme dual to the Hopf algebra $F[t]/(t^2-1)$.

\begin{definition}\label{ZgradTate}
Let $A=\bbZ$ (resp. $\bbZ/2$) and $B$ the multiplicative group $\bbG_m$ 
(resp. $\mu_2$). 
An {\em $A$-grading} on a Tannakian category $\cC$ consists of the following data:
\begin{enumerate}
\item[(i)] A functorial $A$-grading on objects $X=\bigoplus X^a$, compatible with tensor products $(X\otimes Y)^a=\oplus_{a=b+c} X^b \otimes Y^c$;
\item[(ii)] An $A$-grading of the identity functor $id_{\cC}$, compatible with tensor products;
\item[(iii)] a homomorphism $w: B\to \underline{\mathrm{Aut}}^\otimes(\id_{\cC})$;
\item[(iv)] a central homomorphism $B\to \underline{\mathrm{Aut}}^\otimes(\omega)$ for every fiber functor $\omega$. 
\end{enumerate}
\end{definition}

\begin{definition}\label{Tate3ple}
A {\em Tate triple} $\cT=(\cC,w,T)$ consists of a Tannakian category $\cC$,
a $\bbZ$-grading $w: \bbG_m \to \underline{\mathrm{Aut}}^\otimes(\id_{\cC})$ (called the {\em weight grading}), and an invertible object $T$ (called the {\em Tate object}) of weight $-2$. Given an object $X\in \cC$ and an integer $n$, we will write $X(n)$ for $X \otimes T^{\otimes n}$. A {\em $K$-valued fiber functor on $\cT$} (with $K$ a field extension of $F$) is a $K$-valued fiber functor $\omega$ on $\cC$ endowed with an isomorphism $\omega(T)\simeq \omega(T(1))$. If $\cT$ admits a $F$-valued fiber functor, then $\cT$ is called a {\em neutral Tate triple}.
\end{definition}
\begin{example}\label{ex:Tate}
\begin{itemize}
\item[(i)] Let $G$ be an algebraic group scheme endowed with a central homomorphism $w:\bbG_m \to G$ and with a homomorphism $t:G \to \bbG_m$ such that $ t \circ w =-2$. Let $T$ be the representation of $G$ on $F$ such that $g \in G$ acts as multiplication by $t(g)$. Then, as explained in \cite[Example~5.4]{DelMil}, $\cT=(\mathrm{Rep}_F(G),w,T)$ is a neutral Tate triple.

\item[(ii)] Let $k$ be a field of characteristic zero. If we assume that the standard conjecture $C(Z)$ holds for every smooth projective $k$-scheme $Z$ (see \S\ref{sec:Kunneth}), then $\cT=(\Num^\dagger(k)_\bbQ, w, \bbQ(1))$ is a Tate triple; see \cite[Thm.~6.7]{DelMil}. The $\bbZ$-grading $w$ is given by $(Z,p,m)_i := (Z, (p \circ \underline{\pi}_Z^{2m+i})(Z),m)$, where $\underline{\pi}^n_Z$ are algebraic correspondences such that $\overline{H^\ast_{dR}}(\underline{\pi}^n_Z) = \pi^n_Z$. The Tannakian category $\Num^\dagger(k)_\bbQ$ is obtained from $\Num(k)_\bbQ$ by modifying its symmetry isomorphism constraints using the algebraic correspondences $\underline{\pi}_Z^n$. Moreover, if we assume that the standard conjecture $D(Z)$ holds for every smooth projective $k$-scheme $Z$ (see \S\ref{sec:homological}), then de Rham cohomology $H^\ast_{dR}$ give rise to a $k$-valued fiber functor on $\cT$.
\end{itemize}
\end{example}
As explained in \cite[Prop.~5.5]{DelMil}, every Tate triple $\cT=(\cC,w,T)$ gives use to a central homomorphism $w:\bbG_m \to \mathrm{Gal}(\cC)$ and to a homomorphism $t: \mathrm{Gal}(\cC) \to \bbG_m$ such that $t \circ w=-2$.
\begin{lemma}\label{G0lem}
Let $G$ be an affine group scheme endowed with a central homomorphism $w: \bbG_m \to G$
and with a homorphism $t: G \to \bbG_m$ such that $t \circ w=-2$. Consider the kernel 
$G_0:={\rm Ker}(t: G\to \bbG_m)$ and the associated Tannakian category ${\rm Rep}_F(G_0)$. Then, the following four conditions hold:
\begin{enumerate}
\item[(i)] The $\bbZ$-grading on ${\rm Rep}_F(G)$ induces a $\bbZ/2$-grading
$\epsilon: \mu_2 \to G_0$ on ${\rm Rep}_F(G_0)$, making it into a $\bbZ/2$-graded
Tannakian category.
\item[(ii)] The inclusion $G_0 \hookrightarrow G$ gives rise to an (essentially) surjective 
$\otimes$-functor $Q:{\rm Rep}_F(G) \to {\rm Rep}_F(G_0)$, which maps homogeneous objects of weight $n$ to homogeneous objects of weight $n$ (mod $2$) and the Tate object $T$ to the $\otimes$-unit ${\bf 1}$. 
\item[(iii)] The Tate object $T$ becomes an {\em identity object} in ${\rm Rep}_F(G_0)$, \ie $T \simeq T\otimes T$. Moreover, the functor $\tau:  {\rm Rep}_F(G_0) \to
{\rm Rep}_F(G_0)$ given by 
$X \mapsto X\otimes T$ is an equivalence of  categories.
\item[(iv)] Two homogeneous objects $X, Y \in \mathrm{Rep}_F(G)$ of weights $n$ and $m$, respectively, become isomorphic in $\mathrm{Rep}_F(G_0)$ if and only if $m-n=2\ell$ for some $\ell \in \bbZ$ and $X(\ell)\simeq Y$. 
\end{enumerate}
\end{lemma}

\subsection{Quotient categories}{(consult Milne~\cite{Milne2})}\label{appendix:quotient}
\begin{theorem}\label{thm:quot1Tann}
Let $S \subset \cC$ be an inclusion of neutral Tannakian categories with Galois groups ${\rm Gal}(\cS)$ and ${\rm Gal}(\cC)$. Then, the following four conditions hold:
\begin{enumerate}
\item[(i)] There is a
quotient neutral Tannakian category $\cQ$ and an exact $\otimes$-functor $Q: \cC \to \cQ$ such that all objects of $\cQ$ are subquotients of objects in the image of $Q$. Moreover, the objects of $\cS$ are precisely those of $\cC$ which become trivial in $\cQ$. 
\item[(ii)] The inclusion $\cS \subset \cC$ determines a surjective group homomorphism on the
corresponding Galois groups ${\rm Gal}(\cC) \twoheadrightarrow {\rm Gal}(\cS)$.
\item[(iii)] Let $H$ be the kernel of the homomorphism ${\rm Gal}(\cC) \twoheadrightarrow {\rm Gal}(\cS)$.
Then, there is an equivalence of categories ${\rm Rep}_F(H) \simeq  \cQ$. 
\item[(iv)] Given an object $X\in \cC$, one
denotes by $X^H$ the largest subobject of $X$ on which $H$ acts trivially. Under this notation, the subcategory $\cS \subset \cC$ agrees with the Tannakian subcategory $\cC^H\subset \cC$ of those objects $X$ such that $X^H=X$. 
\end{enumerate}
\end{theorem}

\begin{definition}\label{quotTann}
Let $\cC$ be a neutral Tannakian category and $\omega_0:\cS \to \mathrm{Vect}(F)$ a $F$-valued fiber functor on a Tannakian subcategory $\cS \subset \cT$.
Then, the {\em quotient category $\cC/\omega_0$} is the pseudo-abelian envelope of the category $\cC'$ which as the same objects as $\cC$ and morphisms given by
\begin{equation*}
\Hom_{\cC^\prime}(X,Y) := \omega_0(\underline{\Hom}_\cC(X,Y)^H)\,.
\end{equation*}
Here, $\underline{\Hom}_\cC(X,Y)$ stands for the internal Hom-object of the symmetric monoidal category $\cC$ and $H \subset {\rm Gal}(\cC)$ for the subgroup described in Theorem~\ref{thm:quot1Tann}(iii).
\end{definition}

\section{Orbit categories}\label{app:orbit}
In this appendix (which is of general interest) we recall the notion of orbit category and describe its behavior with respect to four distinct operations. In what follows, $F$ will be a field, $(\cC,\otimes, {\bf 1})$ a $F$-linear, additive, rigid symmetric monoidal category, and $\cO \in \cC$ a $\otimes$-invertible object.

As explained in \cite[\S7]{CvsNC}, we can then consider the {\em orbit category} $\cC/_{\!\!-\otimes \cO}$. It has the same objects as $\cC$ and morphisms given by 
\begin{equation}\label{eq:comp}
\Hom_{\cC/_{\!\!-\otimes \cO}}(X,Y) := \bigoplus_{j \in \bbZ} \Hom_\cC(X,Y\otimes \cO^{\otimes j})\,.
\end{equation}
The composition law is induced by the one on $\cC$. By construction, $\cC/_{\!\!-\otimes \cO}$ is $F$-linear, additive, rigid symmetric monoidal (see \cite[Lemma~7.3]{CvsNC}), and comes equipped with a canonical projection $\otimes$-functor $\tau: \cC \to \cC/_{\!\!-\otimes \cO}$. Moreover, $\tau$ is endowed with a natural $2$-isomorphism $\tau \circ (-\otimes \cO) \stackrel{\sim}{\Rightarrow} \tau$ and is $2$-universal among all such functors. 
\subsection*{Change of coefficients}
Recall from \S\ref{sec:coefficients} the change-of-coefficients mechanism.
\begin{lemma}\label{lem:orbit}
Let $K$ be a field extension of $F$. Then, the canonical functor
$$(\cC/_{\!\!-\otimes \cO})\otimes_FK \stackrel{\sim}{\too} (\cC\otimes_FK)/_{\!\!-\otimes \cO}$$
is a $\otimes$-equivalence.
\end{lemma}
\begin{proof}
The proof follows from the above description \eqref{eq:comp} and from the fact that the functor $-\otimes_FK $, from $F$-vector spaces to $K$-vector spaces, commutes with arbitrary sums.
\end{proof}
\subsection*{Pseudo-abelian envelope}
The {\em pseudo-abelian envelope $\cC^\natural$} of $\cC$ is defined as follows: the objects are the pairs $(X,e)$, where $X \in \cC$ and $e$ is an idempotent of the $F$-algebra $\End_\cC(X)$, and the morphisms are given by
$$ \Hom_{\cC^\natural}((X,e),(Y,e')):= e \circ \Hom_\cC(X,Y) \circ e'\,.$$
Composition is naturally induced by $\cC$. The symmetric monoidal structure on $\cC$ extends natural to $\cC^\natural$ by the formula $(X,e) \otimes (Y,e') := (X \otimes Y, e \otimes e')$.
\begin{lemma}\label{lem:orbit-pseudo}
We have a fully-faithful, $F$-linear, additive, $\otimes$-functor
\begin{eqnarray}\label{eq:orbit-pseudo}
\cC^\natural/_{\!\!-\otimes\cO} \too (\cC/_{\!\!-\otimes\cO})^\natural && (X,e) \mapsto (X, \tau(e))\,.
\end{eqnarray}
Moreover, the induced functor $\tau^\natural: \cC^\natural \to (\cC/_{\!\!-\otimes \cO})^\natural$ factors through \eqref{eq:orbit-pseudo}.
\end{lemma}
\begin{proof}
The fact that the functor \eqref{eq:orbit-pseudo} is $F$-linear, additive, and symmetric monoidal is clear. Let us then show that it is moreover fully-faithful. Given objects $(X,e)$ and $(Y,e')$ in $\cC^\natural$, we have the following equality
\begin{equation}\label{eq:Hom1}
\Hom_{\cC^\natural/_{\!\!-\otimes \cO}}((X,e),(Y,e'))= \bigoplus_{j \in \bbZ} e \circ \Hom_\cC(X,Y \otimes \cO^{\otimes j}) \circ (e' \otimes \id_{\cO^{\otimes j}})\,,
\end{equation}
where $e \otimes \id_{\cO^{\otimes j}}$ is an idempotent of $Y \otimes \cO^{\otimes j}$. On the other hand, we have the equality
\begin{equation}\label{eq:Hom2}
\Hom_{(\cC/_{\!\!-\otimes \cO})^\natural}((X,\tau(e)),(Y,\tau(e')))= \tau(e) \circ \bigoplus_{j \in \bbZ} \Hom_\cC(X,Y\otimes \cO^{\otimes j}) \circ \tau(e') \,.
\end{equation}
The composition operation of the orbit category $\cC/_{\!\!-\otimes \cO}$ allows us then to conclude that \eqref{eq:Hom1}=\eqref{eq:Hom2}, which implies that the functor \eqref{eq:orbit-pseudo} is fully-faithful. The fact that the induced functor $\tau^\natural$ factors through \eqref{eq:orbit-pseudo} is by now clear.
\end{proof}
\subsection*{Quotients categories}
\begin{lemma}\label{lem:orbit-kernel}
Let $H: \cC/_{\!\!-\otimes \cO} \to \cD$ be a $F$-linear, additive, $\otimes$-functor. Then, we obtain a full, $F$-linear, additive, $\otimes$-functor
\begin{eqnarray}\label{eq:orbit-kernel}
(\cC/{\it Ker})/_{\!\!-\otimes \cO} \too (\cC/_{\!\!-\otimes \cO})/\mathrm{Ker}(H) && X \mapsto X\,,
\end{eqnarray}
where ${\it Ker}$ denotes the kernel of the composed functor $\cC \stackrel{\tau}{\to} \cC/_{\!\!-\otimes \cO} \stackrel{H}{\to} \cD$. Moreover, the induced functor $\cC/{\it Ker} \to (\cC/_{\!\!-\otimes \cO})/\mathrm{Ker}(H)$ factors through \eqref{eq:orbit-kernel}.
\end{lemma}
\begin{proof}
The fact that the functor \eqref{eq:orbit-kernel} is $F$-linear, additive, and symmetric monoidal is clear. In order to show that \eqref{eq:orbit-kernel} is moreover full, note that every morphism $[\underline{f}]=[\{ f_j\}_{j \in \bbZ}\}] \in \Hom_{(\cC/_{\!\!-\otimes \cO})/\mathrm{Ker}(H)}(X,Y)$, with $ \{ f_j\}_{j \in \bbZ} \in \bigoplus_{j \in \bbZ} \Hom_\cC(X,Y \otimes \cO^{\otimes j})$ admits a a canonical lift given by $\underline{\tilde{f}}=\{[f_j]\}_{j \in \bbZ} \in \Hom_{(\cC/{\it Ker})/_{\!\!-\otimes \cO}}(X,Y)$, with $\{ [f_j]\}_{j \in \bbZ} \in \bigoplus_{j \in \bbZ} \Hom_{\cC/{\it Ker}}(X,Y \otimes \cO^{\otimes j})$. The fact that the induced functor $\cC/{\it Ker} \to (\cC/_{\!\!-\otimes \cO})/\mathrm{Ker}(H)$ factors through \eqref{eq:orbit-kernel} is by now clear.
\end{proof}
\subsection*{Change of symmetry}
Recall from Proposition~\ref{prop:new} the general procedure $(-)^\dagger$ of changing the symmetry isomorphism constraints of a symmetric monoidal category.

\begin{lemma}\label{lem:new-technical}
Let $H: \cC \to \mathrm{sVect}(F)$ be a $F$-linear $\otimes$-functor with values in the category of finite dimensional super $F$-vector spaces. If for every object $X \in \cC$, the K\"unneth projectors $\pi^\pm_X:H(X) \twoheadrightarrow H^\pm(X) \hookrightarrow H(X)$ can be written as $\pi_X^\pm=H(\underline{\pi}_X^\pm)$ with $\underline{\pi}_X^\pm \in \End_\cC(X)$, then the identity functor is a $\otimes$-equivalence
$$ \cC^\dagger/_{\!\!-\otimes \cO} \stackrel{\sim}{\too} (\cC/_{\!\!-\otimes \cO})^\dagger\,.$$ 
\end{lemma}
\begin{proof}
Recall from \cite[Lemma~7.3]{CvsNC} that the symmetry isomorphism constraints $c_{X,Y}$ of the category $\cC/_{\!\!-\otimes \cO}$ are the image of those of $\cC$ under the projection functor $\tau: \cC \to \cC/_{\!\!-\otimes \cO}$. Similarly, the endomorphisms $\underline{\pi}_X^+$ of $\cC/_{\!\!-\otimes \cO}$ are the image of those of $\cC$ under $\tau$. The proof now follows from the fact that the new symmetry isomorphism constraints are given by $c^\dagger_{X,Y}:= c_{X,Y} \circ (e_X \otimes e_Y)$, with $e_X=2 \cdot \underline{\pi}_X^+ - \id_X$, and from the fact that the projection functor $\tau$ is symmetric monoidal. 
\end{proof}

\end{document}